\documentclass[12pt]{article}
\usepackage[]{amsmath,amssymb}
\usepackage{amscd}
\usepackage{latexsym}
\usepackage{cite}
\usepackage{amsthm}

\newtheorem{definition}{Definition}[section]
\newtheorem{theorem}[definition]{Theorem}
\newtheorem{lemma}[definition]{Lemma}

\newtheorem{example}[definition]{Example}

\newtheorem{note}[definition]{Note}

\newtheorem{proposition}[definition]{Proposition}

\begin{document}
\title{\bf 
Using a $q$-shuffle algebra to describe the \\basic module $V(\Lambda_0)$
for the quantized\\  enveloping algebra $U_q(\widehat{\mathfrak{sl}}_2)$}
 \author{
Paul Terwilliger 
}
\date{}

\maketitle
\begin{abstract} We consider the quantized enveloping algebra   $U_q(\widehat{\mathfrak{sl}}_2)$
and its basic module $V(\Lambda_0)$. This module is infinite-dimensional, irreducible, integrable, and highest-weight.
We describe $V(\Lambda_0)$ using a $q$-shuffle algebra in the following way.
Start with the free associative algebra $\mathbb V$ on two generators $x,y$.
The standard basis for 
$\mathbb V$ consists of the words in $x,y$.
In 1995 M. Rosso introduced 
an associative algebra structure on $\mathbb V$, called a
$q$-shuffle algebra.
For $u,v\in \lbrace x,y\rbrace$ their 
$q$-shuffle product is
$u\star v = uv+q^{(u,v)}vu$, where
$( u,v) =2$ 
(resp. $(u,v) =-2$)
if $u=v$ (resp. 
 $u\not=v$).
 Let $\mathbb U$ denote the subalgebra of the $q$-shuffle algebra $\mathbb V$
 that is generated by $x, y$. Rosso showed that the algebra $\mathbb U$ is isomorphic to the positive part of $U_q(\widehat{\mathfrak{sl}}_2)$.
 In our first main result, we turn $\mathbb U$ into a  $U_q(\widehat{\mathfrak{sl}}_2)$-module.
Let $\bf U$ denote the $U_q(\widehat{\mathfrak{sl}}_2)$-submodule of $\mathbb U$ generated by the empty word.
In our second main result, we show that the $U_q(\widehat{\mathfrak{sl}}_2)$-modules $\bf U$ and $V(\Lambda_0)$ are isomorphic.
Let $\bf V$ denote the subspace of $\mathbb V$ spanned by the words that do not begin with $y$ or $xx$.
In our third main result, we show that $\bf U = \mathbb U \cap {\bf V}$.
\bigskip

\noindent
{\bf Keywords}.  Quantized enveloping algebra, $q$-Serre relations, basic module, $q$-shuffle algebra.
\hfil\break
\noindent {\bf 2020 Mathematics Subject Classification}. 
Primary: 17B37. Secondary: 05E14, 81R50.

 \end{abstract}
 
 \section{Introduction}
 The quantized enveloping algebra $U_q(\widehat{\mathfrak{sl}}_2)$ is associative, noncommutative, and infinite dimensional. A presentation by generators and relations is given in Appendix B below.
 The algebra  $U_q(\widehat{\mathfrak{sl}}_2)$ has a subalgebra $U^+_q$, called the positive part. 
 The algebra $U^+_q$
 has a presentation involving two generators $A, B$ and two relations, called the $q$-Serre relations:
 \begin{align*}
\lbrack A, \lbrack A, \lbrack A, B\rbrack_q \rbrack_{q^{-1}} \rbrack =0, \qquad \qquad 
\lbrack B, \lbrack B, \lbrack B, A \rbrack_q \rbrack_{q^{-1}}\rbrack = 0.
\end{align*}

 \noindent Both $U^+_q$ and $U_q(\widehat{\mathfrak{sl}}_2)$  are well known in algebraic combinatorics \cite{ariki,drg,TD00},  representation theory \cite{beck,bcp,charp,xxz}, and mathematical physics \cite{baspp, basXXZ, DaviesPlus, JM}.
In the following paragraphs, we describe a few situations in which $U^+_q$ and $U_q(\widehat{\mathfrak{sl}}_2)$ play a role.
 \medskip

 \noindent 
 There is an object in algebraic combinatorics called a tridiagonal pair \cite{TD00}. Roughly speaking, this  is a pair of diagonalizable linear maps 
 on a finite-dimensional vector space, that each act on the eigenspaces of the other one in a block-tridiagonal fashion.
  According to \cite[Example~1.7]{TD00}, for a finite-dimensional irreducible $U^+_q$-module $V$ on which the generators $A, B$ are not nilpotent,
  the pair $A, B$ act on $V$ as a tridiagonal pair.
 The resulting tridiagonal pair is said to have $q$-geometric type or $q$-Serre type. This type of tridiagonal pair is described in 
 \cite{CurtH, hasan2, hasan,shape, uqsl2hat, nonnil, qtet, drg, TwoRel, yy2}.

 \medskip
 
 

 \noindent Another object in algebraic combinatorics is a  partially ordered set $Y$ called the Young lattice \cite[p.~288]{stanley}. The elements of $Y$ are the Young diagrams (partitions), and the partial order is given by
 diagram inclusion. 
 Define a vector space $V$ consisting of the formal linear combinations of $Y$. In the Hayashi realization \cite[Theorem~10.6]{ariki}
 the vector space $V$ becomes an integrable $U_q(\widehat{\mathfrak{sl}}_2)$-module with the following features.
 Each Young diagram $\lambda$ is a weight vector. The  $U_q(\widehat{\mathfrak{sl}}_2)$-generators $E_0, E_1, F_0, F_1$ act on $\lambda$ as follows.
 Color the boxes of $\lambda$ alternating blue and red, with the top left box colored blue. The generator $E_0$ (resp. $E_1$) sends $\lambda$ to a linear combination of the Young diagrams $\mu$
  obtained from $\lambda$ by removing a blue box (resp. red box). In this linear combination the $\mu$-coefficient is a power of $q$ that depends on the location of the
 box $\lambda/\mu$. Similarly, $F_0$ (resp. $F_1$) sends $\lambda$ to a linear combination of the Young diagrams $\mu$
obtained from $\lambda$ by adding a blue box (resp. red box). In this linear combination the $\mu$-coefficient is a power of $q$ that depends on the location of
 the box $\mu/\lambda$.
 The  $U_q(\widehat{\mathfrak{sl}}_2)$-submodule of $V$ generated by the empty Young diagram  is denoted $V(\Lambda_0)$ and called the basic representation.
 The $U_q(\widehat{\mathfrak{sl}}_2)$-module $V(\Lambda_0)$ is  infinite-dimensional, irreducible, integrable, and highest-weight. For more detail about $V(\Lambda_0)$ see \cite[Chapter~10]{ariki}, \cite[Section~9]{hongkang}, \cite[Chapter~5]{JM}.
 \medskip

\noindent Next we recall an embedding, due to M. Rosso \cite{rosso1, rosso} of $U^+_q$ into a $q$-shuffle algebra.
Start with a free associative algebra $\mathbb V$ on two generators $x,y$. These generators are called letters.
For $n\geq 0$, a word of length $n$ in $\mathbb V$ is a product of letters $\ell_1 \ell_2 \cdots \ell_n$.
We interpret the word of length 0 to be the multiplicative identity in $\mathbb V$; this word is called trivial and denoted by $\bf 1$.
The words in $\mathbb V$ form a basis for the vector space $\mathbb V$; this basis is called standard.
In \cite{rosso1, rosso} M. Rosso introduced 
an associative algebra structure on $\mathbb V$, called a
$q$-shuffle algebra.
For letters $u,v$ their 
$q$-shuffle product is
$u\star v = uv+q^{( u,v) }vu$, where
$( u,v)=2$ 
(resp. $(u,v) =-2$) 
if $u=v$ (resp. 
 $u\not=v$).
 In \cite[Theorem~15]{rosso} Rosso gave an injective algebra homomorphism
$\natural$ from $U^+_q$ into the $q$-shuffle algebra
${\mathbb V}$, that sends $A\mapsto x$ and $B\mapsto y$.
\medskip

\noindent We mention some applications of the map $\natural : U^+_q \to \mathbb V$. 
In \cite{damiani} I. Damiani obtained a   Poincar\'e-Birkhoff-Witt (or PBW) basis for 
 $U^+_q$ whose elements are
 defined recursively.  In \cite[Proposition~6.1]{beck} J. Beck obtained another  PBW basis for $U^+_q$ by adjusting some of the elements in the Damiani PBW basis.
In \cite{catalan} (resp. \cite{catBeck}) we applied the map $\natural $ to the Damiani (resp. Beck) PBW basis, and expressed the images
in the standard basis for $\mathbb V$. We gave the images in closed form \cite[Theorem~1.7]{catalan}, \cite[Theorem~7.1]{catBeck}. The map $\natural$ is used in \cite{alternating} to define the alternating
elements of $U^+_q$. In \cite[Theorem~10.1] {alternating} a set of alternating elements is shown to form a PBW basis for $U^+_q$. This PBW basis is said to be alternating \cite[Definition~10.3]{alternating}. 
In \cite{altCE} we used  the alternating elements
to obtain a central extension $\mathcal U^+_q$ of
$U^+_q$. The algebra $\mathcal U^+_q$ is defined by
generators and relations. These generators, said to be
alternating, are in bijection with the alternating elements of
$U^+_q$. 
By \cite[Lemma~3.3]{altCE} there exists a surjective algebra homomorphism
$\mathcal U^+_q \to U^+_q$ that sends
each alternating generator of 
$\mathcal U^+_q$ to the corresponding alternating element in $U^+_q$.
In \cite[Lemma~3.6]{altCE}  this homomorphism is adjusted to obtain an algebra isomorphism
$\mathcal U_q^+ \to U^+_q \otimes  
\mathbb F \lbrack z_1, z_2,\ldots\rbrack$
where $\mathbb F$ is the ground field and
$\lbrace z_n\rbrace_{n=1}^\infty$ are mutually commuting indeterminates.
By \cite[Theorem~10.2]{altCE} the alternating generators 
form a PBW basis for 
$\mathcal U_q^+$. The algebra $\mathcal U^+_q$ is called the alternating central extension of $U^+_q$ \cite{altCE, compactUqp}.
We remark that $\mathcal U^+_q$ is related
to the work of
Baseilhac, Koizumi, Shigechi concerning the $q$-Onsager
algebra and integrable lattice models \cite{BK05, basnc}.
See \cite{basFMA,
baspp, basXXZ, BKojima, pospart, compactUqp, factorUq, pbwqO, compQons, qOnsACE} for related work.
\medskip

 \noindent Turning to the present paper, our goal is to describe the basic $U_q(\widehat{\mathfrak{sl}}_2)$-module $V(\Lambda_0)$ using the
 $q$-shuffle algebra $\mathbb V$. We have three main results, which are summarized below.
 Let ${\rm End}(\mathbb V)$ denote the algebra  consisting  of the 
linear maps from $\mathbb V$ to $\mathbb V$. We now define some maps $X, Y, K$ in ${\rm End}(\mathbb V)$.
 The map $X$ (resp. $Y$) is the automorphism of the free algebra $\mathbb V$ that sends $x \mapsto q x$ and $y \mapsto y$
  (resp. $x \mapsto  x$ and $y \mapsto q y$).
  Define $K=X^2Y^{-2}$.
 Define the  maps  $A^*_L$, $B^*_L$, $A^*_R$, $B^*_R$ in ${\rm End}(\mathbb V)$ that send ${\bf 1} \mapsto 0$ and for a nontrivial word
$w=\ell_1\ell_2\cdots \ell_n$ in $\mathbb V$,
\begin{align*}
&
A^*_L w =
\ell_2\cdots \ell_n \delta_{\ell_1, x},
\qquad \qquad  \;\;
B^*_L w=
\ell_2\cdots \ell_n \delta_{\ell_1, y},
\\
&
A^*_R w=
\ell_1\cdots \ell_{n-1} \delta_{\ell_n, x},
\qquad \qquad
B^*_Rw =
\ell_1\cdots \ell_{n-1} \delta_{\ell_n, y}.
\end{align*}
Here $\delta_{r,s}$ is the Kronecker delta.
Define the maps  $A_\ell$, $B_\ell$, $A_r$, $B_r$ in ${\rm End}(\mathbb V)$ such that 
for $v \in \mathbb V$,
\begin{align*}
&
A_{\ell} v = x\star v, \quad \qquad
B_{\ell} v = y\star v, \qquad \quad
A_{r} v = v \star x, \quad \qquad
B_{r} v = v \star y. 
\end{align*}
 Let $\mathbb U$ denote the subalgebra of the
 $q$-shuffle algebra $\mathbb V$ that is generated by $x, y$. By construction the map $\natural : U^+_q \to \mathbb U$ is an algebra isomorphism.
 Our first main result is that  $\mathbb U$ becomes a  $U_q(\widehat{\mathfrak{sl}}_2)$-module  on which the  $U_q(\widehat{\mathfrak{sl}}_2)$-generators act as follows:

 \bigskip
 
 \centerline{
 \begin{tabular}{c| ccccccc}
{\rm generator} &  $E_0$ & $F_0$ & $K^{\pm 1}_0$  & $E_1$ & $F_1$ & $K^{\pm 1} _1$  & $D^{\pm 1}$ 
\\
\hline
{\rm action on $ \mathbb U$} & $A^*_R$ & $ \frac{ q A_r K^{-1} - q^{-1} A_{\ell}}{q-q^{-1}}$ & $q^{\pm 1}K^{\mp 1}$  & $B^*_R$ & $ \frac{ B_r K-B_\ell}{q-q^{-1}}$ & $K^{\pm 1}$ &  $X^{\mp 1}$ 
\\
\end{tabular}
}
\bigskip

\noindent Let $\bf U$ denote the submodule of the $U_q(\widehat{\mathfrak{sl}}_2)$-module $\mathbb U$ that is generated by the vector $\bf 1$.
 Our second main result is that the $U_q(\widehat{\mathfrak{sl}}_2)$-modules $\bf U$ and $V(\Lambda_0)$ are isomorphic.
  Let $\bf V$ denote the intersection of  the kernel of $B^*_L$ and the kernel of $(A^*_L)^2$. The vector space ${\bf V}$ has a basis consisting of the
  words in $\mathbb V$ that do not begin with $y$ or $xx$. Note that the sum ${\bf V} = \mathbb F{\bf 1} + \mathbb F x + xy\mathbb V$ is direct.
 Our third main result is that ${\bf U} = \mathbb U \cap {\bf V}$.
 \medskip
 
 \noindent The paper is organized as follows. Section 2 contains some preliminaries.
 In Section 3 we recall the algebra $U^+_q$ and discuss its basic properties.
 In Section 4 we describe the free algebra $\mathbb V$.
 In Section 5 we describe the maps $X,Y,K$ in ${\rm End}(\mathbb V)$.
 In Section 6 we describe the maps $A^*_L, B^*_L, A^*_R, B^*_R$ in ${\rm End}(\mathbb V)$.
 In Section 7 we describe the $q$-shuffle algebra $\mathbb V$.
 In Section 8 we describe the maps $A_\ell, B_\ell, A_r, B_r$ in ${\rm End}(\mathbb V)$.
 In Section 9 we describe the subalgebra $\mathbb U$ of the $q$-shuffle algebra $\mathbb V$.
 In Sections 10, 11 we give our main results, which are
 Theorems \ref{thm:m1}, \ref{thm:3}, \ref{thm:m3}.
 In Section 12 we describe some variations on Theorem \ref{thm:m1}.
 In Appendix A we display some relations that are satisfied by the maps from the main body of the paper.
 In Appendix B we give a presentation of $U_q(\widehat{\mathfrak{sl}}_2)$. 
 In Appendix C we give a basis for some of the weight spaces of the $U_q(\widehat{\mathfrak{sl}}_2)$-module $\bf U$.
 In Appendix D we show how the $U_q(\widehat{\mathfrak{sl}}_2)$-generators act on the bases in Appendix C.
 In Appendix E we discuss a linear algebraic situation that comes up in Section 11.

  \section{Preliminaries} We now begin our formal argument. 
 Recall the natural numbers $\mathbb N = \lbrace 0,1,2,\ldots \rbrace$ and integers $\mathbb Z = \lbrace 0, \pm 1, \pm 2, \ldots \rbrace$.
 Let $\mathbb F$ denote a field with characteristic zero. Throughout this paper, every  vector space we discuss is understood to be over $\mathbb F$.
 Every algebra we discuss is understood to be associative, over $\mathbb F$, and have a multiplicative identity. A subalgebra has the same multiplicative identity
 as the parent algebra.
Let $\mathcal A$ denote an algebra.  An {\it automorphism} of $\mathcal A$ is an algebra isomorphism $\mathcal A \to \mathcal A$.
The {\it opposite algebra} $\mathcal A^{\rm opp}$ consists of the vector space $\mathcal A$ and the multiplication map
$\mathcal A \times \mathcal A \to \mathcal A$, $(a,b)\mapsto ba$.
An {\it antiautomorphism} of $\mathcal A$ is an algebra isomorphism $\mathcal A \to \mathcal A^{\rm opp}$. 
 \medskip
 
 \noindent We recall a few concepts from linear algebra.  Let $V$ denote a vector space, and consider an $\mathbb F$-linear map $T:V \to V$. The map $T$ is said to be {\it nilpotent}
whenever there exists a positive integer $n$ such that $T^n=0$. The map $T$ is said to be {\it locally nilpotent} whenever for all $v \in V$ there exists
a positive integer $n$ such that $T^n v=0$. If $T$ is nilpotent then $T$ is locally nilpotent. If $T$ is locally nilpotent and the dimension of $V$ is finite, then $T$ is nilpotent.
\medskip

   \noindent Throughout the paper, fix a nonzero $q \in \mathbb F$
that is not a root of unity.
Recall the notation
\begin{align*}
\lbrack n\rbrack_q = \frac{q^n-q^{-n}}{q-q^{-1}}
\qquad \qquad n \in \mathbb N.
\end{align*}

\section{The positive part of $U_q(\widehat{\mathfrak{sl}}_2)$}
\noindent Later in the paper, we will discuss the quantized enveloping algebra  $U_q(\widehat{\mathfrak{sl}}_2)$. For now, we consider
a subalgebra $U^+_q$  of $U_q(\widehat{\mathfrak{sl}}_2)$, called the positive part. Shortly  we will give a presentation of $U^+_q$ by generators and relations.
\medskip

\noindent For elements $\mathcal X, \mathcal Y$ in any algebra, define their
commutator and $q$-commutator by 
\begin{align*}
\lbrack \mathcal X, \mathcal Y \rbrack = \mathcal X\mathcal Y-\mathcal Y\mathcal X, \qquad \qquad
\lbrack \mathcal X, \mathcal Y \rbrack_q = q \mathcal X\mathcal Y- q^{-1}\mathcal Y\mathcal X.
\end{align*}
\noindent Note that 
\begin{align}
\label{eq:qs}
\lbrack \mathcal X, \lbrack \mathcal X, \lbrack \mathcal X, \mathcal Y\rbrack_q \rbrack_{q^{-1}} \rbrack
= 
\mathcal X^3\mathcal Y-\lbrack 3\rbrack_q \mathcal X^2 \mathcal Y \mathcal X+ 
\lbrack 3\rbrack_q \mathcal X\mathcal Y \mathcal X^2 -\mathcal Y \mathcal X^3.
\end{align}

\begin{definition} \label{def:U} \rm
(See \cite[Corollary~3.2.6]{lusztig}.) 
Define the algebra $U^+_q$ by generators $A$, $B$ and relations
\begin{align}
\label{eq:qOns1}
&\lbrack A, \lbrack A, \lbrack A, B\rbrack_q \rbrack_{q^{-1}} \rbrack =0,
\qquad \qquad
\lbrack B, \lbrack B, \lbrack B, A\rbrack_q \rbrack_{q^{-1}}\rbrack = 0.
\end{align}
We call $U^+_q$ the {\it positive part of $U_q(\widehat{\mathfrak{sl}}_2)$}.
The relations \eqref{eq:qOns1}  are called the {\it $q$-Serre relations}.
\end{definition}
\noindent We mention some symmetries of $U^+_q$. 

\begin{lemma}
\label{lem:aut} There exists an automorphism $\sigma$ of $U^+_q$ that sends $A \leftrightarrow B$.
Moreover $\sigma^2 = {\rm id}$, where ${\rm id}$ denotes the identity map.
\end{lemma}

\begin{lemma}\label{lem:antiaut} {\rm (See \cite[Lemma~2.2]{catalan}.)}
There exists an antiautomorphism $\dagger$ of $U^+_q$ that fixes each of $A$, $B$.
 Moreover $\dagger^2={\rm id}$.
\end{lemma}

\begin{lemma} {\rm (See \cite[Lemma~3.4]{factorUq}.)} The maps $\sigma$, $\dagger$ commute.
\end{lemma}
 
\begin{definition}\label{def:tauA} \rm Let $\tau$ denote the composition of $\sigma$ and $\dagger$. Note that $\tau$ is an antiautomorphism of $U^+_q$ that sends
$A \leftrightarrow B$. We have $\tau^2 = {\rm id}$.
\end{definition}

\noindent Next we describe a grading for the algebra $U^+_q$.
The $q$-Serre relations are homogeneous in both
$A$ and $B$. Therefore,
the algebra $U^+_q$ has a ${\mathbb N}^2$-grading 
for which $A$ and $B$ are homogeneous,
with degrees $(1,0)$ and $(0,1)$ respectively. For $(r,s)\in \mathbb N^2$ let $U^+_q(r,s)$ denote the $(r,s)$-homogeneous component of the grading. The 
dimension of $U^+_q(r,s)$ is described by a generating function, as we now discuss.  Let $t$ and $u$ denote commuting indeterminates.

\begin{definition}\label{def:RS} \rm Define the generating function
\begin{align*}
\Phi(t,u) = 
\prod_{n=1}^\infty 
\frac{1}{1-t^n u^{n-1}}\,
\frac{1}{1-t^{n} u^{n}}\,
\frac{1}{1-t^{n-1} u^{n}}.
\end{align*}
\noindent Using $(1-z)^{-1}= 1+ z+ z^2 + \cdots$ we expand the above generating function as a power series:
\begin{align*}
\Phi(t,u) = \sum_{(r,s) \in \mathbb N^2} d_{r,s} t^r u^s, \qquad \qquad d_{r,s} \in \mathbb N.
\end{align*}
For notational convenience, define $d_{r,-1}=0$ and $d_{-1,s}=0$ for $r,s\in \mathbb N$.
\end{definition}

\begin{example} 
\label{ex:tabledij}
\rm (See \cite[Example~3.4]{alternating}.)
For $0 \leq r,s\leq 6$ we display $d_{r,s}$
in the $(r,s)$-entry of the matrix below:
\begin{align*}
	\left(
         \begin{array}{ccccccc}
               1&1&1&1&1&1&1 
	       \\
               1&2&3&3&3&3&3 
               \\
	       1&3&6&8&9&9&9 
              \\
	      1&3&8&14&19&21&22 
             \\
	     1&3&9&19&32&42&48 
            \\
	    1&3&9&21&42&66&87
	    \\
	    1&3&9&22&48&87&134
                  \end{array}
\right)
\end{align*}
\end{example}
\noindent We have $\Phi(t,u)= \Phi(u,t)$. Moreover $d_{r,s} = d_{s,r}$ for $(r,s) \in \mathbb N^2$.

\begin{lemma} \label{lem:gfRS}{\rm (See \cite[Definition~3.2, Corollary~3.7]{alternating}.)} 
For $(r,s) \in \mathbb N^2$ we have
\begin{align*}
d_{r,s}= {\rm dim} \,U^+_q(r,s).
\end{align*}
\end{lemma}
\noindent 
Our next goal is to show that $d_{r, s-1}\leq d_{r,s} $ and $d_{r-1,s} \leq d_{r,s}$ for $(r,s)\in \mathbb N^2$. To reach the goal,
we modify the generating function $\Phi(t,u)$ in the following way.
\begin{definition}\label{def:De1} \rm Define the generating function
\begin{align}
\label{eq:DeForm}
\Delta(t,u)= 
\prod_{n=1}^\infty 
\frac{1}{1-t^n u^{n-1}}\,
\frac{1}{1-t^{n} u^{n}}\,
\frac{1}{1-t^{n} u^{n+1}}.
\end{align}
\end{definition}
\begin{lemma}\label{lem:De2} We have $\Delta(t,u) = \Phi(t,u)(1-u)$ and $\Delta(u,t) = \Phi(t,u)(1-t)$. Moreover
\begin{align*}
\Delta(t,u) = \sum_{(r,s) \in \mathbb N^2} \bigl( d_{r,s} - d_{r,s-1}\bigr) t^r u^s, 
\qquad \quad 
\Delta(u,t) = \sum_{(r,s) \in \mathbb N^2} \bigl( d_{r,s} - d_{r-1,s}\bigr) t^r u^s.
\end{align*}
\end{lemma}
\begin{proof} Use Definitions \ref{def:RS}, \ref{def:De1}.
\end{proof}
\begin{lemma} \label{lem:De3} For $(r,s)\in \mathbb N^2$ we have
$d_{r,s-1} \leq d_{r,s}$ and
$d_{r-1,s} \leq d_{r,s}$.
\end{lemma}
\begin{proof} Expand the right-hand side of \eqref{eq:DeForm} as a power series. In this power series, the coefficient of $t^r u^s$ is nonnegative for $(r,s) \in \mathbb N^2$. The result follows in view of
Lemma \ref{lem:De2}.
\end{proof}
\noindent Our next general goal is to compute ${\rm max}\lbrace d_{r,s} | s \in \mathbb N\rbrace $ for $r \in \mathbb N$, and
${\rm max}\lbrace d_{r,s} | r \in \mathbb N\rbrace $ for $s \in \mathbb N$. To reach the goal, we will use the concept of a partition.
\medskip

\noindent For $n \in \mathbb N$, a {\it partition of $n$} is a sequence $\lambda = \lbrace \lambda_i \rbrace_{i=1}^\infty$ of natural numbers such that $\lambda_i \geq \lambda_{i+1}$ for $i\geq 1$ and $n = \sum_{i=1}^\infty \lambda_i$.
Let $p_n$ denote the number of partitions of $n$. For example,
\bigskip
 
 \centerline{
 \begin{tabular}{c|ccccccc}
$n$ &  $0$ & $1$ & $2$ & $3$ &$4$ &$5$ &$6$
\\
\hline
$p_n$ & $1$ &$1$ &$2$&  $3$& $5$& $7$ &  $11$
\end{tabular}
}
\bigskip

\noindent Define the generating function for partitions:
\begin{align}
\label{eq:gfp}
 p(t) = \sum_{n \in \mathbb N} p_n t^n.
 \end{align}
\noindent
 The following result is well known; see for example \cite[Theorem~8.3.4]{bruIntro}.
\begin{align}
p(t) = \prod_{n=1}^\infty \frac{1}{ 1-t^n}.
\label{eq:WK}
\end{align}

\noindent We expand the generating function $\bigl(p(t)\bigr)^3$ as a power series:
\begin{align}
\label{eq:p3}
\bigl(p(t)\bigr)^3=\sum_{n \in \mathbb N} \mu_n t^n, \qquad \qquad \mu_n \in \mathbb N.
\end{align}
Consider the coefficients $\lbrace \mu_n \rbrace_{n \in \mathbb N}$. For example,
\bigskip

 \centerline{
 \begin{tabular}{c|ccccccc}
$n$ &  $0$ & $1$ & $2$ & $3$ &$4$ &$5$ &$6$
\\
\hline
$\mu_n$ & $1$ &$3$ &$9$&  $22$& $51$& $108$ &  $221$
\end{tabular}
}

\begin{proposition} \label{lem:max} For $r \in \mathbb N$ we have 
\begin{align} \label{eq:mS}
\mu_r = {\rm max}\lbrace d_{r,s} | s \in \mathbb N\rbrace.
\end{align}
For $s \in \mathbb N$ we have
\begin{align}\label{eq:mR}
\mu_s = {\rm max}\lbrace d_{r,s} | r \in \mathbb N\rbrace.
\end{align}
\end{proposition}
\begin{proof} First, for $r \in \mathbb N$ we verify \eqref{eq:mS}. Let $\mu'_r$ denote right-hand side of \eqref{eq:mS}. We show that $\mu_r =\mu'_r$.
By Lemma  \ref{lem:De3}, we may view
\begin{align*} 
\mu'_r = \sum_{s \in \mathbb N} (d_{r,s}-d_{r,s-1}).
\end{align*}
By this and Lemma \ref{lem:De2},
\begin{align}
 \Delta(t,1) = \sum_{(r,s)\in \mathbb N^2} \bigl(d_{r,s}-d_{r,s-1}\bigr) t^r =
\sum_{r \in \mathbb N} \mu'_r t^r.
\label{eq:mup}
\end{align}
Set $u=1$ in \eqref{eq:DeForm}, and evaluate the result using \eqref{eq:WK}, \eqref{eq:p3}. This yields
\begin{align}
\label{eq:mup2}
\Delta(t,1)  = \prod_{n=1}^\infty \frac{1}{(1-t^n)^3} = \bigl( p(t) \bigr)^3 = \sum_{r \in \mathbb N} \mu_r t^r.
\end{align}
Comparing \eqref{eq:mup}, \eqref{eq:mup2} we obtain $\mu_r =\mu'_r$ for $r \in \mathbb N$. We have verified \eqref{eq:mS}. 
The second assertion in the proposition statement follows from the first assertion in the proposition statement and the comment above  Lemma \ref{lem:gfRS}.
\end{proof}

\noindent Our next general goal is to embed $U^+_q$ into a $q$-shuffle algebra. For this $q$-shuffle algebra the underlying vector space is a free algebra on two generators.
This free algebra is described in the next section.

 \section{The free algebra $\mathbb V$}

 \noindent Let $x$, $y$ denote noncommuting indeterminates. Let $\mathbb V$ denote the free algebra with generators $x,y$.
By a {\it letter} in $\mathbb V$ we mean $x$ or $y$. For $n \in \mathbb N$,  a {\it word of length $n$}  in $\mathbb V$ is a product of
letters $\ell_1 \ell_2 \cdots \ell_n$. We interpret the word of length 0 to be the multiplicative identity in $\mathbb V$; this word is called {\it trivial} and denoted by $\bf 1$.
The vector space $\mathbb V$ has a basis consisting of its words; this basis is called {\it standard}.
 \medskip
 
\noindent We mention some symmetries of the free algebra $\mathbb V$. For the next four lemmas, the proofs are routine and omitted.
\begin{lemma}
\label{lem:faut} There exists an automorphism $\sigma$ of the free algebra $\mathbb V$ that sends $x \leftrightarrow y$.
Moreover $\sigma^2 = {\rm id}$.
\end{lemma}

\begin{lemma}\label{lem:fantiaut} 
There exists an antiautomorphism $\dagger$ of the free algebra $\mathbb V$ that fixes each of $x$, $y$.
 Moreover $\dagger^2={\rm id}$.
\end{lemma}

\begin{lemma} The map $\sigma$ from Lemma \ref{lem:faut} commutes with the map $\dagger$ from Lemma \ref{lem:fantiaut}.
\end{lemma}
 
\begin{lemma}\label{def:ftauA} There exists an antiautomorphism $\tau$ of the free algebra $\mathbb V$ that sends
$x \leftrightarrow y$. The map $\tau$ is the composition the map $\sigma$ from  Lemma \ref{lem:faut}
 and the map $\dagger$ from Lemma \ref{lem:fantiaut}. 
We have $\tau^2 = {\rm id}$. 
\end{lemma}

\begin{example}\label{ex:aaa} \rm The automorphism $\sigma$ sends
\begin{align*}
xxx \leftrightarrow yyy, \qquad xxyy\leftrightarrow yyxx, \qquad xyxxyy \leftrightarrow yxyyxx.
\end{align*}
\noindent The antiautomorphism $\dagger$ sends
\begin{align*}
xxx \leftrightarrow xxx, \qquad xxyy\leftrightarrow yyxx, \qquad xyxxyy \leftrightarrow yyxxyx.
\end{align*}
The antiautomorphism $\tau$ sends
\begin{align*}
xxx \leftrightarrow yyy, \qquad xxyy\leftrightarrow xxyy, \qquad xyxxyy \leftrightarrow xxyyxy.
\end{align*}
\end{example}

\noindent The free algebra $\mathbb V$ has a $\mathbb N^2$-grading
for which $x$ and $y$ are homogeneous, with degrees
$(1,0)$ and $(0,1)$  respectively. For 
 $(r,s)\in \mathbb N^2$ 
let $\mathbb V(r,s)$ denote the $(r,s)$-homogeneous component of the grading.
These homogeneous components are described as follows.
Let $w=\ell_1\ell_2\cdots \ell_n$ denote a word in $\mathbb V$.
The {\it $x$-degree} of $w$ is the
cardinality of the set
$\lbrace i | 1 \leq i \leq n,\;\ell_i = x\rbrace$.
The {\it $y$-degree} of $w$ is the
cardinality of the set
$\lbrace i | 1 \leq i \leq n,\;\ell_i = y\rbrace$.
For $(r,s)\in \mathbb N^2$ 
the subspace $\mathbb V(r,s)$ has a basis consisting
of the words in $\mathbb V$ that have $x$-degree $r$ and
$y$-degree $s$. The dimension of $\mathbb V(r,s)$ is equal to
the binomial coefficient $\binom{r+s}{r}$. By construction $\mathbb V(0,0)=\mathbb F {\bf 1}$. By construction, the
sum $\mathbb V=\sum_{(r,s) \in \mathbb N^2}  \mathbb V(r,s)$ is direct.
\medskip

 
 \begin{example}\label{ex:v2e} \rm The following is a basis for the vector space $\mathbb V(2,3)$:
 \begin{align*}
 &xx yyy, \quad xyxyy, \quad xyyxy, \quad xyyyx, \quad yxxyy, 
 \\
 & yxyxy, \quad yxyyx, \quad yyxxy, \quad yyxyx, \quad yyyxx.
 \end{align*}
 \end{example}
 
\noindent Let ${\rm End}(\mathbb V)$ denote the algebra  consisting  of the 
$\mathbb F$-linear maps from $\mathbb V$ to $\mathbb V$. Let $I$ denote the identity in ${\rm End}(\mathbb V)$.

 \section{The maps $X$, $Y$, $K$}
 \noindent In this section we describe some maps $X$, $Y$, $K$  in ${\rm End}(\mathbb V)$ that will be used in our main results.
 
 \begin{definition}\label{def:xy} \rm Let $X$ denote the automorphism of the free algebra $\mathbb V$ that sends $x \mapsto q x$ and $y \mapsto y$.
  Let $Y$ denote the automorphism of the free algebra $\mathbb V$ that sends $x \mapsto  x$ and $y \mapsto q y$.
 \end{definition}

\begin{example}\label{ex:xy} \rm The map $X$ sends
\begin{align*}
xxx \mapsto q^3xxx, \qquad xxyy\mapsto q^2 xxyy, \qquad xyxxyy \mapsto q^3xyxxyy.
\end{align*}
\noindent The map $Y$ sends
\begin{align*}
xxx \mapsto xxx, \qquad xxyy\mapsto q^2 xxyy, \qquad xyxxyy \mapsto q^3 xyxxyy.
\end{align*}
\end{example}
 
\begin{lemma} \label{lem:xyE} For $(r,s) \in \mathbb N^2$ the maps $X$ and $Y$ act on
$\mathbb V(r,s)$ as $q^r I$ and $q^s I$, respectively.
\end{lemma}
\begin{proof} By the description of $\mathbb V(r,s)$  above Example \ref{ex:v2e}.
\end{proof}

 \noindent By construction the maps $X$, $Y$ are invertible, and they commute. 
 
  \begin{definition}\label{def:k} \rm Define $K=X^2 Y^{-2}$.
   Thus $K$ is the automorphism of the free algebra $\mathbb V$ that sends $x \mapsto  q^2x$ and $y \mapsto q^{-2} y$. 
 \end{definition}

\begin{example}\label{ex:k} \rm The map $K$ sends
\begin{align*}
xxx \mapsto q^6 xxx, \qquad xxyy\mapsto xxyy, \qquad xyxxyy \mapsto xyxxyy.
\end{align*}
\end{example}
 
\begin{lemma} \label{lem:k} For $(r,s) \in \mathbb N^2$ the map $K$ acts on
$\mathbb V(r,s)$ as $q^{2r-2s} I$.
\end{lemma}
\begin{proof} By Lemma \ref{lem:xyE} and Definition \ref{def:k}.
\end{proof}

\begin{lemma} \label{lem:AAXYK} The following diagrams commute:
\begin{equation*}
{\begin{CD}
\mathbb V @>X^{\pm 1}  >>
               {\mathbb V}
              \\
         @V \sigma VV                   @VV \sigma V \\
         \mathbb V @>>Y^{\pm 1} >
                                  {\mathbb  V}
                        \end{CD}} \qquad \qquad
 {\begin{CD}
\mathbb V @>Y^{\pm 1}  >>
               {\mathbb V}
              \\
         @V \sigma VV                   @VV \sigma V \\
         \mathbb V @>>X^{\pm 1} >
                                  {\mathbb  V}
                        \end{CD}} \qquad \qquad                                           
{\begin{CD}
\mathbb V @>K^{\pm 1}  >>
               {\mathbb V}
              \\
         @V \sigma VV                   @VV \sigma V \\
         \mathbb V @>>K^{\mp 1} >
                                  {\mathbb  V}
                        \end{CD}} \qquad 
  \end{equation*}
  \begin{equation*}
{\begin{CD}
\mathbb V @>X^{\pm 1}  >>
               {\mathbb V}
              \\
         @V \dagger VV                   @VV \dagger V \\
         \mathbb V @>>X^{\pm 1} >
                                  {\mathbb  V}
                        \end{CD}} \qquad \qquad
                        {\begin{CD}
                     \mathbb V @>Y^{\pm 1}  >>
               {\mathbb V}
              \\
         @V \dagger VV                   @VV \dagger V \\
         \mathbb V @>>Y^{\pm 1} >
                                  {\mathbb  V}
                        \end{CD}} \qquad \qquad                    
{\begin{CD}
\mathbb V @>K^{\pm 1}  >>
               {\mathbb V}
              \\
         @V \dagger VV                   @VV \dagger V \\
         \mathbb V @>>K^{\pm 1} >
                                  {\mathbb  V}
                        \end{CD}} \qquad 
                        \end{equation*}
      \begin{equation*}
{\begin{CD}
\mathbb V @>X^{\pm 1}  >>
               {\mathbb V}
              \\
         @V \tau VV                   @VV \tau V \\
         \mathbb V @>>Y^{\pm 1} >
                                  {\mathbb  V}
                        \end{CD}} \qquad \qquad
                        {\begin{CD}
                        \mathbb V @>Y^{\pm 1}  >>
               {\mathbb V}
              \\
         @V \tau VV                   @VV \tau V \\
         \mathbb V @>>X^{\pm 1} >
                                  {\mathbb  V}
                        \end{CD}} \qquad \qquad                     
{\begin{CD}
\mathbb V @>K^{\pm 1}  >>
               {\mathbb V}
              \\
         @V \tau VV                   @VV \tau V \\
         \mathbb V @>>K^{\mp 1} >
                                  {\mathbb  V}
                        \end{CD}} \qquad
                        \end{equation*}
 \end{lemma}
 \begin{proof} Routine.
 \end{proof}

 \section{The maps $A^*_L$, $B^*_L$, $A^*_R$, $B^*_R$}
\noindent In this section we recall from \cite{boxq} some maps $A^*_L$, $B^*_L$, $A^*_R$, $B^*_R$ in ${\rm End}(\mathbb V)$ that will be used in our main results.
First we mention some notation. The Kronecker delta $\delta_{r,s}$ is equal to $1$ if $r=s$, and $0$ if $r \not=s$.

\begin{definition} \label{def:AABB} \rm (See \cite[Lemma~4.3]{boxq}.)  Define the  maps  $A^*_L$, $B^*_L$, $A^*_R$, $B^*_R$ in ${\rm End}(\mathbb V)$ as follows.
For a nontrivial word
$w=\ell_1\ell_2\cdots \ell_n$ in $\mathbb V$,
\begin{align*}
&
A^*_L w =
\ell_2\cdots \ell_n \delta_{\ell_1, x},
\qquad \qquad  \;\;
B^*_L w=
\ell_2\cdots \ell_n \delta_{\ell_1, y},
\\
&
A^*_R w=
\ell_1\cdots \ell_{n-1} \delta_{\ell_n, x},
\qquad \qquad
B^*_Rw =
\ell_1\cdots \ell_{n-1} \delta_{\ell_n, y}.
\end{align*}
Moreover
\begin{align}
A^*_L {\bf 1}=0, \qquad \quad 
B^*_L{\bf 1}=0, \qquad \quad 
A^*_R{\bf 1}=0, \qquad \quad 
B^*_R{\bf 1}=0.
\label{eq:ABaction}
\end{align}
\end{definition}
\begin{example}\rm The maps $A^*_L$, $B^*_L$, $A^*_R$, $B^*_R$ are illustrated in the table below.
\bigskip

\centerline{
\begin{tabular}[t]{c|cccccc}
$w$ & $x$ & $y$ & $xx$ & $xy$ & $yx$ & $yy$ \\
    \hline  
   $A^*_Lw$ & $\bf 1$ & $0$ & $x$ & $y$ & $0$ & $0$ \\
   $B^*_L w$ & $0$ & $\bf 1$ & $0$ & $0$ & $x$ & $y$ \\
   $A^*_Rw$ & $\bf 1$ & $0$ & $x$ & $0$ & $y$ & $0$ \\
   $B^*_Rw$ & $0$ & $\bf 1$ & $0$ & $x$ & $0$ & $y$ 
         \\
              \end{tabular}
              }
              \bigskip

\noindent
\end{example}

\begin{lemma} \label{lem:aug}
For $v \in \mathbb V$,
\begin{align*}
&A^*_L(xv) = v,\qquad  \quad
A^*_L(yv) = 0,\qquad  \quad
B^*_L(xv) = 0,\qquad \quad
B^*_L(yv) = v,
\\
&A^*_R(vx) = v,\qquad \quad
A^*_R(vy) = 0,\qquad \quad
B^*_R(vx) = 0,\qquad \quad
B^*_R(vy) = v.
\end{align*}
\end{lemma}
\begin{proof} Use Definition  \ref{def:AABB}.
\end{proof}

\noindent For notational convenience, define $\mathbb V(r,-1)=0$ and $\mathbb V(-1,s)=0$ for $r,s \in \mathbb N$.
\begin{lemma} \label{lem:down} For $(r,s)\in \mathbb N^2$ we have
\begin{align*} 
&A^*_L \mathbb V(r,s) \subseteq \mathbb V(r-1,s), \qquad \qquad B^*_L \mathbb V(r,s) \subseteq \mathbb V(r,s-1),
\\
&A^*_R \mathbb V(r,s) \subseteq  \mathbb V(r-1,s), \qquad \qquad B^*_R \mathbb V(r,s) \subseteq \mathbb V(r,s-1).
\end{align*}
\end{lemma}
\begin{proof} By Definition \ref{def:AABB} or Lemma \ref{lem:aug}.
\end{proof}
\begin{lemma} \label{lem:ln} The maps $A^*_L$, $B^*_L$, $A^*_R$, $B^*_R$ are locally nilpotent on the vector space $\mathbb V$.
\end{lemma}
\begin{proof} We mentioned above Example \ref{ex:v2e}
that the sum $\mathbb V = \sum_{(r,s) \in \mathbb N^2} \mathbb V(r,s)$ is direct. The result follows from this
and Lemma  \ref{lem:down}.
\end{proof}
\noindent Next we describe how the maps $X$, $Y$ are related to the maps  $A^*_L$, $B^*_L$, $A^*_R$, $B^*_R$.

\begin{lemma} \label{lem:XYAABB}
We have
\begin{align*}
&X A^*_L = q^{-1}A^*_L X, \qquad 
X B^*_L = B^*_L X, \qquad 
X A^*_R = q^{-1}A^*_R X, \qquad 
X B^*_R= B^*_R X, \\
&YA^*_L =A^*_L Y, \qquad 
Y B^*_L = q^{-1}B^*_L Y, \qquad 
Y A^*_R = A^*_R Y, \qquad 
Y B^*_R= q^{-1}B^*_R Y.
\end{align*}
\end{lemma} 
\begin{proof} By Lemmas \ref{lem:xyE},  \ref{lem:down}.
\end{proof}

\noindent The next result is about $A^*_L$ and $B^*_L$; a similar result
holds for $A^*_R$ and $B^*_R$.
Observe that the sum $\mathbb V = \mathbb F {\bf 1} + x \mathbb V + y \mathbb V$ is direct.

\begin{lemma}
\label{lem:ABKer} The following {\rm (i)--(v)} hold:
\begin{enumerate}
\item[\rm (i)] ${\rm ker}A^*_L$ has a basis consisting of the words in $\mathbb V$ that do not begin with $x$;
\item[\rm (ii)] ${\rm ker} A^*_L = \mathbb F {\bf 1} + y \mathbb V$;
\item[\rm (iii)] ${\rm ker}B^*_L$ has a basis consisting of the words in $\mathbb V$ that do not begin with $y$;
\item[\rm (iv)] ${\rm ker} B^*_L = \mathbb F {\bf 1} + x \mathbb V$;
\item[\rm (v)] ${\rm ker} A^*_L \cap {\rm ker} B^*_L = \mathbb F {\bf 1}$.
\end{enumerate}
\end{lemma}
\begin{proof} Use Definition \ref{def:AABB} and the observation above the lemma statement.
\end{proof}

\noindent The following result appears in \cite{boxq}; we give a short proof for the sake of completeness.
\begin{lemma} 
\label{lem:ABirred} {\rm (See \cite[Lemma~4.6]{boxq}.)}
Let $W$ denote a nonzero subspace of $\mathbb V$ that is closed under
$A^*_L$ and $B^*_L$. Then ${\bf 1} \in W$.
\end{lemma}
\begin{proof} For $n \in \mathbb N$ define $\mathbb V_n = \sum_{r+s\leq n} \mathbb V(r,s)$. Note that $\mathbb V_0 = \mathbb F {\bf 1}$.
We have $\mathbb V_{n-1} \subseteq \mathbb V_n$ for $n\geq 1$, and $\mathbb V = \cup_{n \in \mathbb N} \mathbb V_n$. 
For $n\geq 1$ we have $A^*_L \mathbb V_n \subseteq \mathbb V_{n-1}$ and $B^*_L \mathbb V_n \subseteq \mathbb  V_{n-1}$, in view of Lemma \ref{lem:down}.
Since $W \not=0$,  there exists $n \in \mathbb N$ such that $W \cap \mathbb V_n \not=0$. Assume for the moment that $n=0$. Then ${\bf 1}\in W$ and we are done.
Next assume that $n\geq 1$. Without loss, we may assume that $W \cap \mathbb V_{n-1} =0$.
Pick $0 \not=v \in W \cap \mathbb V_n$. We have  $A^*_L v \in W \cap \mathbb V_{n-1}=0$ and 
 $B^*_L v \in W \cap \mathbb V_{n-1}=0$, so $v \in \mathbb F {\bf 1}$ in view of Lemma
 \ref{lem:ABKer}(v). By construction $0 \not=v \in W$, so ${\bf 1} \in W$.
\end{proof}

\begin{lemma} 
\label{lem:ABirred2}
Let $W$ denote a nonzero subspace of $\mathbb V$ that is closed under
$A^*_R$ and $B^*_R$. Then ${\bf 1} \in W$.
\end{lemma}
\begin{proof} The $\dagger$-image $W^\dagger$ is a subspace of $\mathbb V$ that is invariant under $A^*_L$ and $B^*_L$. We have ${\bf 1} \in W^\dagger$ by Lemma \ref{lem:ABirred}, and ${\bf 1}^\dagger = {\bf 1}$ 
by construction, so ${\bf 1} \in W$.
\end{proof}

\begin{lemma} \label{lem:AAAABB} The following diagrams commute:
\begin{equation*}
 {\begin{CD}
\mathbb V @>A^*_L >>
               {\mathbb V}
              \\
         @V \sigma VV                   @VV \sigma V \\
         \mathbb V @>>B^*_L >
                                  {\mathbb  V}
                        \end{CD}} 
                        \qquad \qquad
                        {\begin{CD}
\mathbb V @>B^*_L >>
               {\mathbb V}
              \\
         @V \sigma VV                   @VV \sigma V \\
         \mathbb V @>>A^*_L >
                                  {\mathbb  V}
                        \end{CD}} \qquad \qquad 
 {\begin{CD}
\mathbb V @>A^*_R >>
               {\mathbb V}
              \\
         @V \sigma VV                   @VV \sigma V \\
         \mathbb V @>>B^*_R >
                                  {\mathbb  V}
                        \end{CD}} 
                        \qquad \qquad
                        {\begin{CD}
\mathbb V @>B^*_R >>
               {\mathbb V}
              \\
         @V \sigma VV                   @VV \sigma V \\
         \mathbb V @>>A^*_R >
                                  {\mathbb  V}
                        \end{CD}} \qquad
                                    \end{equation*}
                                    \begin{equation*}
       {\begin{CD}
\mathbb V @>A^*_L >>
               {\mathbb V}
              \\
         @V \dagger VV                   @VV \dagger V \\
         \mathbb V @>>A^*_R >
                                  {\mathbb  V}
                        \end{CD}} 
                        \qquad \qquad
                        {\begin{CD}
\mathbb V @>B^*_L >>
               {\mathbb V}
              \\
         @V \dagger VV                   @VV \dagger V \\
         \mathbb V @>>B^*_R >
                                  {\mathbb  V}
                        \end{CD}} \qquad \qquad
                                {\begin{CD}
\mathbb V @>A^*_R >>
               {\mathbb V}
              \\
         @V \dagger VV                   @VV \dagger V \\
         \mathbb V @>>A^*_L >
                                  {\mathbb  V}
                        \end{CD}} 
                        \qquad \qquad
                        {\begin{CD}
\mathbb V @>B^*_R >>
               {\mathbb V}
              \\
         @V \dagger VV                   @VV \dagger V \\
         \mathbb V @>>B^*_L >
                                  {\mathbb  V}
                        \end{CD}} \qquad
                                    \end{equation*}
     \begin{equation*}                                         {\begin{CD}
\mathbb V @>A^*_L >>
               {\mathbb V}
              \\
         @V \tau VV                   @VV \tau V \\
         \mathbb V @>>B^*_R >
                                  {\mathbb  V}
                        \end{CD}} 
                        \qquad \qquad
                        {\begin{CD}
\mathbb V @>B^*_L >>
               {\mathbb V}
              \\
         @V \tau VV                   @VV \tau V \\
         \mathbb V @>>A^*_R >
                                  {\mathbb  V}
                        \end{CD}} \qquad \qquad 
          {\begin{CD}
\mathbb V @>A^*_R >>
               {\mathbb V}
              \\
         @V \tau VV                   @VV \tau V \\
         \mathbb V @>>B^*_L >
                                  {\mathbb  V}
                        \end{CD}} 
                        \qquad \qquad
                        {\begin{CD}
\mathbb V @>B^*_R >>
               {\mathbb V}
              \\
         @V \tau VV                   @VV \tau V \\
         \mathbb V @>>A^*_L >
                                  {\mathbb  V}
                        \end{CD}} \qquad
                                    \end{equation*}

\end{lemma}
\begin{proof} Routine.
\end{proof}

 \section{The $q$-shuffle algebra $\mathbb V$}
 
 \noindent In the previous sections we discussed the free algebra $\mathbb V$. There is
another algebra structure on $\mathbb V$,
called the $q$-shuffle algebra.
This algebra was introduced by Rosso
\cite{rosso1, rosso} and described further by Green
\cite{green}. We will adopt the approach of 
\cite{green}, which is suited to our purpose.
The $q$-shuffle product 
is denoted by $\star$. To describe this product, we first consider some special cases.
We have ${\bf 1} \star v = v \star {\bf 1} = v$ for $v \in \mathbb V$.
 For
letters $u,v$ we have $
u \star v = uv + vu q^{(u,v)}$, where
\bigskip

\centerline{
\begin{tabular}[t]{c|cc}
$(\,,\,)$ & $x$ & $y$
   \\  \hline
   $x$ &
   $2$ & $-2$
     \\
     $y$ &
      $-2$ & $2$
         \\
              \end{tabular}
              }
              \medskip

\noindent
 Thus
   \begin{align*}
    &x \star x = (1+q^2) xx,
 \qquad \qquad \quad
   x \star y = xy + q^{-2}yx,
\\
&y \star x = yx+ q^{-2}xy, \qquad \qquad \quad
    y\star y = (1+q^2)yy.
\end{align*}
\noindent For a letter $u$ and
  a nontrivial word $v= v_1v_2\cdots v_n$ in $\mathbb V$,
  \begin{align}
  \label{eq:Xcv2}
  &u \star v =
  \sum_{i=0}^n v_1 \cdots v_{i} u v_{i+1} \cdots v_n
  q^{
  ( v_1, u)+
  ( v_2, u)+
  \cdots + ( v_{i}, u)},
  \\
  &v \star u = \sum_{i=0}^n v_1 \cdots v_{i} u v_{i+1} \cdots v_n
  q^{
  ( v_{n},u)
  +
  ( v_{n-1},u)
  +
  \cdots
  +
  (v_{i+1},u)
  }.
 \label{eq:vcX2}
  \end{align}
   For example
  \begin{align*}
     & y\star (xxx)= yxxx+ q^{-2} xyxx+ q^{-4} xxyx+q^{-6}xxxy,
    \\
    &(xxx)\star y = q^{-6}yxxx+ q^{-4} xyxx+ q^{-2} xxyx+xxxy.
      \end{align*}
\noindent For nontrivial words $u=u_1u_2\cdots u_r$
and $v=v_1v_2\cdots v_s$ in $\mathbb V$,
\begin{align*}
&u \star v  = u_1\bigl((u_2\cdots u_r) \star v\bigr)
+ v_1\bigl(u \star (v_2 \cdots v_s)\bigr)
q^{
( u_1, v_1) +
( u_2, v_1) +
\cdots
+
(u_r, v_1)},
\\
&u\star v =
\bigl(u \star (v_1 \cdots v_{s-1})\bigr)v_s +
\bigl((u_1 \cdots u_{r-1}) \star v\bigr)u_r
q^{
( u_r, v_1) +
(u_r, v_2) + \cdots +
(u_r, v_s)
}.
\end{align*}
 For example, assume $r=2$ and $s=2$. Then
\begin{align*}
 u \star v &= 
               u_1 u_2 v_1 v_2 
            +
               u_1 v_1 u_2 v_2  q^{( u_2, v_1)}
	     + 
      u_1 v_1 v_2 u_2  q^{( u_2, v_1)+ ( u_2,v_2)}
	  + 
      v_1 u_1 u_2 v_2  q^{(u_1, v_1)+ ( u_2,v_1)}
             \\
	     &\quad + 
      v_1 u_1 v_2 u_2  q^{( u_1, v_1)+ 
       (u_2,v_1)+( u_2,v_2)}
	     + 
      v_1 v_2 u_1 u_2  q^{( u_1, v_1)+ 
       ( u_1,v_2)+(u_2,v_1)+ ( u_2,v_2)}.
\end{align*}
\noindent The map $\sigma$ from Lemma \ref{lem:faut} is an automorphism of the $q$-shuffle algebra $\mathbb V$.
The map $\dagger$ from Lemma \ref{lem:fantiaut} is an antiautomorphism of the $q$-shuffle algebra $\mathbb V$.
The map $\tau$ from Lemma \ref{def:ftauA} is an antiautomorphism of the $q$-shuffle algebra $\mathbb V$.
Above Example \ref{ex:v2e} 
 we mentioned an $\mathbb N^2$-grading of the free algebra $\mathbb V$. This is also an $\mathbb N^2$-grading for the $q$-shuffle algebra $\mathbb V$.
 \medskip
 
 \noindent 
See \cite{grosse, leclerc,negut, boxq, catalan, alternating, altCE} for more information about the $q$-shuffle algebra $\mathbb V$.

\section{The maps $A_\ell$, $B_\ell$, $A_r$, $B_r$}

\noindent In this section we recall from \cite{boxq} some maps $A_\ell$, $B_\ell$, $A_r$, $B_r$ in ${\rm End}(\mathbb V)$ that will be used in our main results.

\begin{definition}
\label{def:Aell} 
\rm (See \cite[Definition~7.1]{boxq}.) 
Define the maps  $A_\ell$, $B_\ell$, $A_r$, $B_r$ in ${\rm End}(\mathbb V)$ as follows.
For $v \in \mathbb V$,
\begin{align*}
&
A_{\ell} v = x\star v, \quad \qquad
B_{\ell} v = y\star v, \qquad \quad
A_{r} v = v \star x, \quad \qquad
B_{r} v = v \star y. 
\end{align*}
\end{definition}

\begin{example}\rm The maps    $A_\ell$, $B_\ell$, $A_r$, $B_r$      are illustrated in the table below.
\bigskip

\centerline{
\begin{tabular}[t]{c|cccc}
$v$ & $\bf 1$ & $x$ & $y$ & $xy$  \\
    \hline  
   $A_\ell v$ & $x$ & $q \lbrack 2 \rbrack_q xx$ & $xy+q^{-2}yx$ & $q\lbrack 2 \rbrack_q xxy+xyx$   \\
   $B_\ell v$ & $y$ & $q^{-2}xy+yx$ & $q \lbrack 2 \rbrack_q yy$ & $q^{-1}\lbrack 2 \rbrack_q xyy+yxy$   \\
   $A_r v$ & $x$ & $q \lbrack 2 \rbrack_q xx$ & $q^{-2}xy+yx$ & $q^{-1} \lbrack 2 \rbrack_q xxy + xyx$   \\
   $B_r v$ & $y$ & $xy+q^{-2}yx$ & $q \lbrack 2 \rbrack_q yy$ & $q \lbrack 2 \rbrack_q  xyy+ yxy$ 
         \\
              \end{tabular}
              }
              \bigskip

\noindent
\end{example}

\begin{lemma} \label{lem:ddown} For $(r,s)\in \mathbb N^2$ we have
\begin{align*} 
&A_\ell \mathbb V(r,s) \subseteq \mathbb V(r+1,s), \qquad \qquad B_\ell \mathbb V(r,s) \subseteq \mathbb V(r,s+1),
\\
&A_r \mathbb V(r,s) \subseteq \mathbb V(r+1,s), \qquad \qquad B_r \mathbb V(r,s) \subseteq \mathbb V(r,s+1).
\end{align*}
\end{lemma}
\begin{proof} By Definition \ref{def:Aell} and the description of $\mathbb V(r,s)$ above Example \ref{ex:v2e}.
\end{proof}

\noindent Next we describe how the maps $X$, $Y$ are related to the maps  $A_\ell$, $B_\ell$, $A_r$, $B_r$.

\begin{lemma} \label{lem:XYAABB2}
We have
\begin{align*}
&X A_\ell= qA_\ell X, \qquad 
X B_\ell = B_\ell X, \qquad 
X A_r = qA_r X, \qquad 
X B_r= B_r X, \\
&YA_\ell =A_\ell Y, \qquad 
Y B_\ell = qB_\ell Y, \qquad 
Y A_r = A_r Y, \qquad 
Y B_r= qB_r Y.
\end{align*}
\end{lemma} 
\begin{proof} By Lemmas \ref{lem:xyE},  \ref{lem:ddown}.
\end{proof}

\begin{lemma} \label{lem:AAAABB2}
The following diagrams commute:
  \begin{equation*}
    {\begin{CD}
\mathbb V @>A_{\ell} >>
               {\mathbb V}
              \\
         @V \sigma VV                   @VV \sigma V \\
         \mathbb V @>>B_{\ell} >
                                  {\mathbb  V}
                        \end{CD}} 
                        \qquad \qquad
                        {\begin{CD}
\mathbb V @>B_{\ell} >>
               {\mathbb V}
              \\
         @V \sigma VV                   @VV \sigma V \\
         \mathbb V @>>A_{\ell} >
                                  {\mathbb  V}
                        \end{CD}} \qquad \qquad
{\begin{CD}
\mathbb V @>A_r  >>
               {\mathbb V}
              \\
         @V \sigma VV                   @VV \sigma V \\
         \mathbb V @>>B_r >
                                  {\mathbb  V}
                        \end{CD}} \qquad \qquad
{\begin{CD}
\mathbb V @>B_r  >>
               {\mathbb V}
              \\
         @V \sigma VV                   @VV \sigma V \\
         \mathbb V @>>A_r >
                                  {\mathbb  V}
                        \end{CD}} \qquad 
                                    \end{equation*}
  \begin{equation*}
    {\begin{CD}
\mathbb V @>A_{\ell} >>
               {\mathbb V}
              \\
         @V \dagger VV                   @VV \dagger V \\
         \mathbb V @>>A_r >
                                  {\mathbb  V}
                        \end{CD}} 
                        \qquad \qquad
                        {\begin{CD}
\mathbb V @>B_{\ell} >>
               {\mathbb V}
              \\
         @V \dagger VV                   @VV \dagger V \\
         \mathbb V @>>B_r >
                                  {\mathbb  V}
                        \end{CD}} \qquad \qquad 
{\begin{CD}
\mathbb V @>A_r  >>
               {\mathbb V}
              \\
         @V \dagger VV                   @VV \dagger V \\
         \mathbb V @>>A_{\ell} >
                                  {\mathbb  V}
                        \end{CD}} \qquad \qquad
{\begin{CD}
\mathbb V @>B_r  >>
               {\mathbb V}
              \\
         @V \dagger VV                   @VV\dagger V \\
         \mathbb V @>>B_{\ell} >
                                  {\mathbb  V}
                        \end{CD}} \qquad                                     \end{equation*}
  \begin{equation*}
      {\begin{CD}
\mathbb V @>A_{\ell} >>
               {\mathbb V}
              \\
         @V \tau VV                   @VV \tau V \\
         \mathbb V @>>B_r >
                                  {\mathbb  V}
                        \end{CD}} 
                        \qquad \qquad
                        {\begin{CD}
\mathbb V @>B_{\ell} >>
               {\mathbb V}
              \\
         @V \tau VV                   @VV \tau V \\
         \mathbb V @>>A_r >
                                  {\mathbb  V}
                        \end{CD}} \qquad \qquad 
{\begin{CD}
\mathbb V @>A_r  >>
               {\mathbb V}
              \\
         @V \tau VV                   @VV \tau V \\
         \mathbb V @>>B_{\ell} >
                                  {\mathbb  V}
                        \end{CD}} \qquad \qquad
{\begin{CD}
\mathbb V @>B_r  >>
               {\mathbb V}
              \\
         @V \tau VV                   @VV\tau V \\
         \mathbb V @>>A_{\ell} >
                                  {\mathbb  V}
                        \end{CD}} \qquad 
                                    \end{equation*}
                           
\end{lemma}
\begin{proof} Routine.
\end{proof}

 \section{The subspace $\mathbb U$}
   \noindent In this section we discuss a subspace $\mathbb U \subseteq \mathbb V$ that will be used in our main results.
   
\begin{definition}\label{def:Usub} \rm Let $\mathbb U$ denote the subalgebra of the $q$-shuffle algebra $\mathbb V$ that is generated by $x,y$. 
\end{definition}
 \noindent 
The algebra $\mathbb U$ is described as follows.
By
\cite[Theorem~13]{rosso1} or
\cite[p.~10]{green},
\begin{align*}
&
x \star x \star x \star y -
\lbrack 3 \rbrack_q
x \star x\star y \star x +
\lbrack 3 \rbrack_q
x \star y \star x \star x -
y \star x \star x \star x  = 0,
\\
&
y \star y \star y \star x -
\lbrack 3 \rbrack_q
y \star y \star x \star y +
\lbrack 3 \rbrack_q
y \star x \star y \star y -
x \star y \star y \star y  = 0.
\end{align*}
So in the $q$-shuffle algebra $\mathbb V$ the elements
$x,y$ satisfy the
$q$-Serre relations.
Therefore, there exists an algebra homomorphism
$\natural$ from $U^+_q$ to the $q$-shuffle algebra $\mathbb V$,
that sends $A\mapsto x$ and $B\mapsto y$.
The map $\natural$ has image $\mathbb U$
by Definition
\ref{def:Usub}, 
and is injective by
  \cite[Theorem~15]{rosso}. 
Consequently $\natural: U^+_q \to \mathbb U$ is an algebra isomorphism. 
By construction  the following diagrams commute:

\begin{equation}
{\begin{CD}
U^+_q @>\natural  >>
               {\mathbb V}
              \\
         @V \sigma VV                   @VV \sigma V \\
         U^+_q @>>\natural >
                                  {\mathbb  V}
                        \end{CD}} \qquad \qquad
{\begin{CD}
U^+_q @>\natural  >>
               {\mathbb V}
              \\
         @V \dagger VV                   @VV \dagger V \\
         U^+_q @>>\natural >
                                  {\mathbb  V}
                        \end{CD}} \qquad \qquad
                        {\begin{CD}
U^+_q @>\natural  >>
               {\mathbb V}
              \\
         @V \tau VV                   @VV \tau V \\
         U^+_q @>>\natural >
                                  {\mathbb  V}
                        \end{CD}} \qquad 
                        \label{eq:threeDiag}
                                    \end{equation}
\noindent Consequently $\mathbb U$ is invariant under each of $\sigma$, $\dagger$, $\tau$.
Earlier we mentioned an $\mathbb N^2$-grading for 
both the algebra $U^+_q$ and the $q$-shuffle algebra $\mathbb V$.
These gradings are related as follows.
The algebra $\mathbb U$ has an $\mathbb N^2$-grading inherited from
$U^+_q$ via $\natural$. With respect to this grading,
for
$(r,s)\in \mathbb N^2$ the $(r,s)$-homogeneous component of
$\mathbb U$ is the 
$\natural$-image of the $(r,s)$-homogeneous component
of $U^+_q$. We denote this homogeneous component by $\mathbb U(r,s)$. By construction,
\begin{align}
\label{eq:VUU}
\mathbb U(r,s)=
\mathbb V(r,s) \cap \mathbb U, \qquad \qquad (r,s) \in \mathbb N^2.
\end{align}
By construction $\mathbb U(0,0)=\mathbb F {\bf 1}$. 
By construction, the sum $\mathbb U = \sum_{(r,s)\in \mathbb N^2} \mathbb U(r,s)$ is direct.
By Lemma \ref{lem:gfRS} and the construction,
\begin{align}
\label{eq:dimu}
d_{r,s} = {\rm dim}\,\mathbb U(r,s), \qquad \qquad (r,s) \in \mathbb N^2.
\end{align}

\begin{lemma} \label{lem:XYKU} For $(r,s) \in \mathbb N^2$ the following hold on $\mathbb U(r,s)$:
\begin{align*}
X = q^r I, \qquad \quad Y = q^s I, \qquad \quad K=q^{2r-2s}I.
\end{align*}
\end{lemma} 
\begin{proof} By Lemmas \ref{lem:xyE}, \ref{lem:k} and since $\mathbb U(r,s) \subseteq \mathbb V(r,s)$.
\end{proof}

\begin{lemma} \label{lem:Uinv} The vector space $\mathbb U$
is invariant under each of
\begin{align*}
X^{\pm 1}, \qquad Y^{\pm 1}, \qquad K^{\pm 1}.
\end{align*}
\end{lemma}
\begin{proof} By Lemma   \ref{lem:XYKU}  and since $\mathbb U=\sum_{(r,s)\in \mathbb N^2} \mathbb U(r,s)$.   \end{proof}

\begin{lemma} \label{lem:Uinv2} {\rm (See \cite[Proposition~9.1]{boxq}.)} The vector space $\mathbb U$
is invariant under each of
\begin{align*}
 A^*_L, \qquad B^*_L, \qquad A^*_R, \qquad B^*_R.
\end{align*}
\end{lemma}

\noindent For notational convenience, define $\mathbb U(r,-1)=0$ and $\mathbb U(-1,s)=0$ for $r,s \in \mathbb N$.
\begin{lemma} \label{lem:down2} For $(r,s)\in \mathbb N^2$ we have
\begin{align*} 
&A^*_L \mathbb U(r,s) \subseteq \mathbb U(r-1,s), \qquad \qquad B^*_L \mathbb U(r,s) \subseteq \mathbb U(r,s-1),
\\
&A^*_R \mathbb U(r,s) \subseteq  \mathbb U(r-1,s), \qquad \qquad B^*_R \mathbb U(r,s) \subseteq \mathbb U(r,s-1).
\end{align*}
\end{lemma}
\begin{proof} By \eqref{eq:VUU} and Lemmas \ref{lem:down}, \ref{lem:Uinv2}.
\end{proof}

\begin{lemma} \label{lem:obv} The subspace $\mathbb U$ is invariant under each of 
\begin{align*}
 A_\ell,\qquad B_\ell,\qquad  A_r, \qquad B_r.
\end{align*}
\end{lemma}
\begin{proof} By Definitions \ref{def:Usub}, \ref{def:Aell}.
\end{proof}
\begin{lemma} \label{lem:ddown2} For $(r,s)\in \mathbb N^2$ we have
\begin{align*} 
&A_\ell \mathbb U(r,s) \subseteq \mathbb U(r+1,s), \qquad \qquad B_\ell \mathbb U(r,s) \subseteq \mathbb U(r,s+1),
\\
&A_r \mathbb U(r,s) \subseteq \mathbb U(r+1,s), \qquad \qquad B_r \mathbb U(r,s) \subseteq \mathbb U(r,s+1).
\end{align*}
\end{lemma}
\begin{proof} By \eqref{eq:VUU} and Lemmas \ref{lem:ddown}, \ref{lem:obv}.
\end{proof}

\noindent In \cite[Propositions~9.1, 9.3]{boxq} there are many relations satisfied by 
the  maps
\begin{align*}
K, \quad K^{-1},
\quad
A^*_L, 
\quad
B^*_L, 
\quad
A^*_R, 
\quad
B^*_R, 
\quad
A_\ell, 
\quad
B_\ell, 
\quad
A_r, 
\quad
B_r.
\end{align*}
\noindent For convenience we reproduce these relations in Appendix A. These relations will be used in our main results.

\section{The $U_q(\widehat{\mathfrak{sl}}_2)$-module $\mathbb U$ and its submodule $\bf U$}
\noindent We now bring in the algebra $U_q(\widehat{\mathfrak{sl}}_2)$. The definition of this algebra can be found in Appendix B. In the present section,
we turn the vector space $\mathbb U$ into a $U_q(\widehat{\mathfrak{sl}}_2)$-module, and describe the submodule $\bf U$ generated by the vector $\bf 1$.
\medskip

\noindent The following is  our first main result.
 \begin{theorem} \label{thm:m1} The vector space $\mathbb U$ becomes a  $U_q(\widehat{\mathfrak{sl}}_2)$-module  on which the  $U_q(\widehat{\mathfrak{sl}}_2)$-generators act as follows:

 \bigskip
 
 \centerline{
 \begin{tabular}{c| ccccccc}
{\rm generator} &  $E_0$ & $F_0$ & $K^{\pm 1}_0$  & $E_1$ & $F_1$ & $K^{\pm 1} _1$  & $D^{\pm 1}$ 
\\
\hline
{\rm action on $ \mathbb U$} & $A^*_R$ & $ \frac{ q A_r K^{-1} - q^{-1} A_{\ell}}{q-q^{-1}}$ & $q^{\pm 1}K^{\mp 1}$  & $B^*_R$ & $ \frac{ B_r K-B_\ell}{q-q^{-1}}$ & $K^{\pm 1}$ &  $X^{\mp 1}$ 
\\
\end{tabular}
}
\end{theorem}
\begin{proof} This is routinely checked using Lemmas 
\ref{lem:Uinv}, \ref{lem:Uinv2}, \ref{lem:obv} along with 
 the relations in Lemmas \ref{lem:XYAABB}, \ref{lem:XYAABB2} and Appendix A. Among the things to check, is that $q A_r K^{-1} - q^{-1} A_{\ell}$ and $B_r K-B_\ell$ satisfy the
 $q$-Serre relations. This can be checked easily using \cite[Lemma~10.3, Corollary~10.4]{boxq}.
\end{proof}

\noindent Consider the $U_q(\widehat{\mathfrak{sl}}_2)$-module $\mathbb U$ from Theorem \ref{thm:m1}.
Recall the $\mathbb N^2$-grading of  $\mathbb U$ from around  \eqref{eq:VUU}. Next we describe how the $U_q(\widehat{\mathfrak{sl}}_2)$-generators act on the homogeneous components of this grading.

\begin{lemma} \label{lem:fdown} For $(r,s)\in \mathbb N^2$ the following hold on $\mathbb U(r,s)$:
\begin{align} 
& K_0=q^{2s-2r+1}I, \qquad \qquad  K_1=q^{2r-2s}I, 
\qquad \qquad D=q^{-r}I.
\label{eq:3}
\end{align}
Moreover 
\begin{align}
&E_0 \mathbb U(r,s) \subseteq \mathbb U(r-1,s), \qquad \qquad F_0 \mathbb U(r,s) \subseteq \mathbb U(r+1,s), \label{eq:4}
\\
& E_1 \mathbb U(r,s) \subseteq \mathbb U(r,s-1), \qquad \qquad F_1 \mathbb U(r,s) \subseteq \mathbb U(r,s+1). \label{eq:5}
\end{align}
\end{lemma}
\begin{proof} By Lemmas  \ref{lem:XYKU},  \ref{lem:down2}, \ref{lem:ddown2} and the data in Theorem \ref{thm:m1}.
\end{proof}

\noindent In this paragraph we recall a few concepts about $U_q(\widehat{\mathfrak{sl}}_2)$-modules; see for example \cite[Section~3.2]{hongkang}. Let $W$ denote a $U_q(\widehat{\mathfrak{sl}}_2)$-module. A {\it weight space} for $W$ is a common eigenspace for the action of
$K_0, K_1, D$ on $W$. The sum of these weight spaces is direct.
We call $W$ a {\it weight module} whenever $W$ is equal to the sum of its weight spaces.
If $W$ is a weight module, then every submodule of $W$ is a weight module \cite[Proposition~3.2.1]{hongkang}.
\smallskip

\noindent We return our attention to the $U_q(\widehat{\mathfrak{sl}}_2)$-module $\mathbb U$ from Theorem  \ref{thm:m1}.
We mentioned above Lemma \ref{lem:XYKU}
that the sum $\mathbb U = \sum_{(r,s)\in \mathbb N^2} \mathbb U(r,s)$ is direct. By this and
\eqref{eq:3},  the $U_q(\widehat{\mathfrak{sl}}_2)$-module $\mathbb U$ is a weight module, and its weight spaces are
the nonzero subspaces among $\mathbb U(r,s)$ ($r,s \in \mathbb N$). Note that these weight spaces have finite dimension.
\medskip

\noindent It turns out that the $U_q(\widehat{\mathfrak{sl}}_2)$-module $\mathbb U$ is not irreducible. Next we consider its submodules.
\begin{definition}\label{def:Lambda}
Let $\bf U$ denote the submodule of the $U_q(\widehat{\mathfrak{sl}}_2)$-module $\mathbb U$ that is generated by the vector $\bf 1$.
\end{definition}
\begin{lemma} \label{lem:tp} For the $U_q(\widehat{\mathfrak{sl}}_2)$-module $\bf U$,
\begin{enumerate}
\item[\rm (i)]  $\bf U$ is contained in every nonzero submodule of the
 $U_q(\widehat{\mathfrak{sl}}_2)$-module $\mathbb U$;
 \item[\rm (ii)] $\bf U$ is the unique irreducible submodule of the  $U_q(\widehat{\mathfrak{sl}}_2)$-module $\mathbb U$.
 \end{enumerate}
 \end{lemma} 
 \begin{proof} (i) Let $W$ denote a nonzero submodule of the  $U_q(\widehat{\mathfrak{sl}}_2)$-module $\mathbb U$.
 The vector space $W$ is invariant under $A^*_R$, $B^*_R$ by Theorem \ref{thm:m1}, so ${\bf 1} \in W$ by Lemma \ref{lem:ABirred2}. The  $U_q(\widehat{\mathfrak{sl}}_2)$-module $\bf U$ is generated by $\bf 1$, so ${\bf U} \subseteq W$.
 \\
 \noindent (ii) By (i) above.
 \end{proof}

\noindent Next we consider how the $U_q(\widehat{\mathfrak{sl}}_2)$-generators act on the vector $\bf 1$.

\begin{lemma} \label{lem:one} For  the $U_q(\widehat{\mathfrak{sl}}_2)$-module $\mathbb U$,
\begin{align*}
& K_0 {\bf1} = q {\bf1}, \qquad K_1{\bf 1}={\bf1},\qquad D{\bf 1} = {\bf1}, \\
&E_0 {\bf 1} = 0,  \qquad 
 F^2_0 {\bf1}=0,
\qquad E_1{\bf 1} =0, \qquad  F_1{\bf 1}=0.
\end{align*}
\end{lemma}
\begin{proof} This is routinely checked using the data in Theorem \ref{thm:m1} and Lemma \ref{lem:fdown}.
\end{proof}

\noindent There is a well known  $U_q(\widehat{\mathfrak{sl}}_2)$-module $V(\Lambda_0)$ that is said to be basic; see \cite[p.~221]{hongkang}.
The module $V(\Lambda_0)$ is 
highest weight, integrable, and level one; see \cite[Chapter ~10]{ariki} and \cite[Chapter~5]{JM}.
The module $V(\Lambda_0)$ is characterized as follows.
\begin{lemma} \label{lem:char} {\rm (See \cite[pp.~63, 64]{JM}.)} There exists a $U_q(\widehat{\mathfrak{sl}}_2)$-module $V(\Lambda_0)$ with the following property: $V(\Lambda_0)$ 
is generated by a nonzero vector $\bf v$
such that 
\begin{align*}
& K_0 {\bf v} = q {\bf v}, \qquad K_1{\bf v}={\bf v}, \qquad D{\bf v} = {\bf v},\\
&E_0 {\bf v} = 0,  \qquad 
 F^2_0 {\bf v}=0,
\qquad E_1{\bf v} =0, \qquad  F_1{\bf v}=0.
\end{align*}
Moreover $V(\Lambda_0)$  is irreducible, infinite-dimensional, and  unique up to isomorphism of $U_q(\widehat{\mathfrak{sl}}_2)$-modules.
\end{lemma}

\noindent The following is our second main result.
\begin{theorem} \label{thm:3} The $U_q(\widehat{\mathfrak{sl}}_2)$-modules $\bf U$ and $V(\Lambda_0)$ are isomorphic.
\end{theorem}
\begin{proof} By Definition \ref{def:Lambda} and Lemmas \ref{lem:one}, \ref{lem:char}.
\end{proof}

\noindent Descriptions of $V(\Lambda_0)$ can be found in  \cite[Chapter ~10]{ariki} and  \cite[Section~9]{hongkang} and \cite[Chapter~5]{JM}; see also \cite[Section~20.4]{carter} and \cite[Chapter~14]{kac}.
Our next general goal is to describe $V(\Lambda_0)$ from the point of view of $\bf U$.

\begin{definition}\label{def:bfU} \rm
For $(r,s)\in \mathbb N^2$ define ${\bf U}(r,s) = {\bf U} \cap \mathbb U(r,s)$. 
\end{definition}
\noindent The  $U_q(\widehat{\mathfrak{sl}}_2)$-module $\bf U$ is a weight module, and its weight spaces are
the nonzero subspaces among ${\bf U}(r,s)$ ($r,s \in \mathbb N$). More detail is given in the next result.

\begin{lemma} \label{lem:fdown2} The following {\rm (i)--(iv)} hold for the $U_q(\widehat{\mathfrak{sl}}_2)$-module $\bf U$.
\begin{enumerate} \item[\rm (i)]  ${\bf U}(0,0)=\mathbb F {\bf 1}$.
\item[\rm (ii)] The sum ${\bf U}  =\sum_{(r,s)\in \mathbb N^2} {\bf U}(r,s)$ is direct.
\item[\rm (iii)] 
For $(r,s)\in \mathbb N^2$ the following hold on ${\bf U}(r,s)$:
\begin{align*} 
& K_0=q^{2s-2r+1}I, \qquad \qquad  K_1=q^{2r-2s}I, 
\qquad \qquad D=q^{-r}I.
\end{align*}
\item[\rm (iv)] For $(r,s) \in \mathbb N^2$,
\begin{align*}
&E_0 {\bf U}(r,s) \subseteq {\bf U}(r-1,s), \qquad \qquad F_0 {\bf U}(r,s) \subseteq {\bf U}(r+1,s), 
\\
& E_1 {\bf U}(r,s) \subseteq {\bf U}(r,s-1), \qquad \qquad F_1 {\bf U}(r,s) \subseteq {\bf U}(r,s+1),
\end{align*}
where ${\bf U}(r,-1)=0$ and ${\bf U}(-1,s)=0$.
\end{enumerate}
\end{lemma}
\begin{proof} (i) By Definition \ref{def:bfU} and since $\mathbb U(0,0)=\mathbb F {\bf 1}$.
\\ \noindent  (ii) Since $\bf U$ is a weight module.
\\ \noindent
(iii), (iv) By Lemma \ref{lem:fdown} and Definition \ref{def:bfU}.
\end{proof}

\noindent Our next goal is to describe the weight space dimensions for the $U_q(\widehat{\mathfrak{sl}}_2)$-module $\bf U$. Recall the partition numbers $\lbrace p_n\rbrace_{n \in \mathbb N}$ from Section 3.
\medskip

\begin{proposition} \label{prop:hk} For $(r,s) \in \mathbb N^2$ the vector space ${\bf U}(r,s)\not=0$ if and only if $r \geq (r-s)^2$. In this case
${\rm dim}\, {\bf U}(r,s)=p_n$, where $n=r-(r-s)^2$.
\end{proposition} 
\begin{proof} For the $U_q(\widehat{\mathfrak{sl}}_2)$-module $V(\Lambda_0)$  the weight space dimensions are described in \cite[pp.~221,~222]{hongkang}. The result
follows from that description and Theorem \ref{thm:3} 
above.
\end{proof}
 \begin{example} 
\label{newtable}
\rm
For $0 \leq r,s\leq 6$ the dimension of ${\bf U}(r,s)$
is given in the $(r,s)$-entry of the matrix below:
\begin{align*}
	\left(
         \begin{array}{ccccccc}
               1&0&0&0&0&0&0 
	       \\
               1&1&1&0&0&0&0 
               \\
	       0&1&2&1&0&0&0
              \\
	      0&0&2&3&2&0&0
             \\
	     0&0&1&3&5&3&1 
            \\
	    0&0&0&1&5&7&5
	    \\
	    0&0&0&0&2&7&11
                  \end{array}
\right)
\end{align*}
\end{example}
\noindent Compare the above matrix with the one in Example \ref{ex:tabledij}.
\medskip

\noindent Next we describe the generating function $\sum_{(r,s)\in \mathbb N^2} {\rm dim} \,{\bf U}(r,s) t^r u^s$.
\medskip

\noindent Define the generating function 
\begin{align}
\label{eq:GFphi}
\phi(t,u)&=  \sum_{n\in \mathbb Z}  t^{n^2} u^{n^2-n}. 
\end{align}
\noindent Note that
\begin{align*}
\phi(t,u)&=  1 + t + tu^2 + t^4 u^2 + t^4 u^6 + \cdots.
\end{align*}

\begin{proposition} We have 
\begin{align*}
 \sum_{(r,s)\in \mathbb N^2} {\rm dim}{\bf U}(r,s) t^r u^s = p(tu) \phi(t,u),
\end{align*}
where $p(t)$ is from \eqref{eq:gfp} and $\phi(t,u)$ is from \eqref{eq:GFphi}.
\end{proposition}
\begin{proof} This is a reformulation of Proposition \ref{prop:hk}.
\end{proof}

\noindent In Appendix C, we give a basis for each nonzero ${\bf U}(r,s)$ such that   $r+s\leq 10$.

 \section{A characterization of the $U_q(\widehat{\mathfrak{sl}}_2)$-module $\bf U$} 
 \noindent In order to motivate this section, we glance at the basis vectors displayed in Appendix C.
 Each displayed vector is a linear combination of some words in $\mathbb V$ that do not begin with $y$ or $xx$. Consequently, each displayed vector is contained in the kernel of $B^*_L$ and the kernel of $(A^*_L)^2$.
 Using this observation, we will characterize the $U_q(\widehat{\mathfrak{sl}}_2)$-module $\bf U$.

 \begin{definition}\label{def:bfv} \rm Let $\bf V$ denote the intersection of  the kernel of $B^*_L$ and the kernel of $(A^*_L)^2$. Note that $\bf V$
 is a subspace of the vector space $\mathbb V$.
 \end{definition}
\noindent We have several comments about $\bf V$.
\begin{lemma} \label{lem:noyxx} The vector space $\bf V$ has a basis consisting of the words in $\mathbb V$ that do not begin with $y$ or $xx$.
\end{lemma}
\begin{proof} By Definitions  \ref{def:AABB}, \ref{def:bfv}.
\end{proof} 

\begin{lemma} \label{lem:fff} The sum ${\bf V} = \mathbb F{\bf 1} + \mathbb F x + xy\mathbb V$ is direct.
\end{lemma}
\begin{proof} By Lemma \ref{lem:noyxx}.
\end{proof}

\noindent We are going to show that ${\bf U}=\mathbb U \cap {\bf V}$. We will do this in several steps. In the first step, we show that $\mathbb U\cap {\bf V}$ is a submodule of the $U_q(\widehat{\mathfrak{sl}}_2)$-module $\mathbb U$.

\begin{lemma}\label{lem:bfVinv} The vector space $\bf V$ is invariant under each of
\begin{align*} 
X^{\pm 1}, \qquad Y^{\pm 1}, \qquad K^{\pm 1}, \qquad A^*_R, \qquad B^*_R.
\end{align*}
\end{lemma}
\begin{proof} Use Definitions \ref{def:xy}, \ref{def:k} 
and Lemma \ref{lem:fff}.
\end{proof} 

\begin{lemma} \label{lem:rrll} The vector space $\bf V$ is invariant under each of
\begin{align*} 
q A_r K^{-1} - q^{-1} A_{\ell}, \qquad \qquad B_r K - B_{\ell}.
\end{align*}
\end{lemma}
\begin{proof} We will use Lemma \ref{lem:fff}.
The map $q A_r K^{-1} - q^{-1} A_{\ell}$ sends ${\bf 1} \mapsto (q-q^{-1})x$ and $x \mapsto 0$. The map $B_r K - B_{\ell}$ sends
${\bf 1} \mapsto 0$ and 
\begin{align*}
x \mapsto q^2 x \star y- y\star x =  q^2 (xy+q^{-2}yx) -(yx+q^{-2}xy) =  (q^2-q^{-2})xy.
\end{align*}
Pick $(r,s) \in \mathbb N^2$ and a word $w \in \mathbb V(r,s)$. The map  $q A_r K^{-1} - q^{-1} A_{\ell}$ sends
\begin{align}
xyw \mapsto q^{1+ 2s-2r} (xyw)\star x - q^{-1} x\star (xyw).
\label{eq:send}
\end{align}
Using \eqref{eq:vcX2} we obtain
\begin{align}
\label{eq:xywx}
(xyw)\star x &= xy(w \star x) + \lbrack 2 \rbrack_q q^{2r-2s-1} xxyw.
\end{align}
Using \eqref{eq:Xcv2} we obtain
\begin{align} \label{eq:xxyw}
x\star (xyw) &= q \lbrack 2 \rbrack_q xxyw + xy(x \star w).
\end{align}
By \eqref{eq:send}--\eqref{eq:xxyw}
the map $q A_r K^{-1} - q^{-1} A_{\ell}$ sends
\begin{align*}
xyw \mapsto xy \bigl( q^{1+ 2s-2r} w\star x - q^{-1} x \star w\bigr).
\end{align*}
\noindent The map  $B_r K - B_{\ell}$ sends
\begin{align}
\label{eq:Sxyw}
xyw \mapsto q^{2r-2s} (xyw)\star y-y\star (xyw).
\end{align}
Using \eqref{eq:vcX2} we obtain
\begin{align}
\label{eq:xywy}
(xyw)\star y &= xy(w \star y) +  q^{2s-2r+2} xyyw + q^{2s-2r} yxyw.
\end{align}
Using \eqref{eq:Xcv2} we obtain
\begin{align} \label{eq:yxyw}
y\star (xyw) &= yxyw + q^{-2} xyyw+ xy(y\star w).
\end{align}
By \eqref{eq:Sxyw}--\eqref{eq:yxyw}
the map $B_r K - B_{\ell}$ sends
\begin{align*}
xyw \mapsto xy \bigl( (q^2-q^{-2})yw+ q^{2r-2s} w\star y -  y \star w\bigr).
\end{align*}
The result follows from the above comments.
\end{proof}

\begin{lemma}\label{lem:uCapv} The vector space $\mathbb U \cap {\bf V}$ is a  submodule of the  $U_q(\widehat{\mathfrak{sl}}_2)$-module $\mathbb U$.
\end{lemma}
\begin{proof} By Theorem  \ref{thm:m1} 
and Lemmas \ref{lem:bfVinv}, \ref{lem:rrll}.
\end{proof}

\begin{lemma}\label{lem:partway} The vector space ${\bf U}$ is a submodule of the  $U_q(\widehat{\mathfrak{sl}}_2)$-module $\mathbb U \cap {\bf V}$.
\end{lemma}
\begin{proof} 
We have $\bf 1 \in \mathbb U$ by the comment below \eqref{eq:VUU}.
We have ${\bf 1} \in {\bf V}$ by Lemma  \ref{lem:fff}.
So ${\bf 1} \in \mathbb U \cap {\bf V}$. The result follows in view of
Definition \ref{def:Lambda} and Lemma \ref{lem:uCapv}.
\end{proof}

\noindent The  $U_q(\widehat{\mathfrak{sl}}_2)$-module $\mathbb U \cap {\bf V}$ is a weight module, and its weight spaces are the nonzero subspaces among
$\mathbb U(r,s) \cap {\bf V}$ $(r,s \in \mathbb N)$. 
We will return to the  $U_q(\widehat{\mathfrak{sl}}_2)$-module $\mathbb U \cap {\bf V}$ after some comments about $\mathbb U$.
\begin{lemma} \label{lem:eeln}
The $U_q(\widehat{\mathfrak{sl}}_2)$-generators $E_0$, $E_1$ are locally nilpotent on the $U_q(\widehat{\mathfrak{sl}}_2)$-module $\mathbb U$.
\end{lemma}
\begin{proof} By  Lemma  \ref{lem:ln} and Theorem \ref{thm:m1}.
\end{proof}

\begin{lemma} \label{lem:ce}
 The $U_q(\widehat{\mathfrak{sl}}_2)$-generators $F_0$, $F_1$ are not locally nilpotent on the  $U_q(\widehat{\mathfrak{sl}}_2)$-module $\mathbb U$. 
 \end{lemma}
 \begin{proof}
 The words $xx$ and $y$ are contained in $\mathbb U$, but
$F^n_0 (xx) \not=0$ and $F^n_1 y\not=0$ for all $n \in \mathbb N$. 
\end{proof}

\begin{lemma} \label{lem:ffln}
The $U_q(\widehat{\mathfrak{sl}}_2)$-generators $F_0$, $F_1$ are locally nilpotent on the $U_q(\widehat{\mathfrak{sl}}_2)$-module $\mathbb U \cap {\bf V}$.
\end{lemma}
\begin{proof} First consider $F_0$. Assume that there exists $v \in \mathbb U \cap {\bf V}$ such that $F_0^{n+1}v \not=0$ for all $n \in \mathbb N$. We will get a contradiction.
By our comments below Lemma \ref{lem:partway},
we may assume without loss of generality that $v \in \mathbb U(r,s)$ for some $(r,s) \in \mathbb N^2$. We have $r\geq 1$ and $s\geq 1$, by  Lemma \ref{lem:fff} and $F_0^2 {\bf 1} = 0$ and $F_0 x = 0$.
Let $n \in \mathbb N$. By \eqref{eq:4} and Lemma \ref{lem:uCapv},
 we have $F^{n+1}_0 v \in \mathbb U(r+n+1,s) \cap {\bf V}$. In particular $F^{n+1}_0 v \in {\bf V}$, so 
$(A^*_L)^2 F^{n+1}_0v = 0$ and $B^*_L F^{n+1}_0 v = 0$ in view of Definition \ref{def:bfv}. We have $A^*_L F_0^{n+1}v\not=0$, by Lemma \ref{lem:ABKer}(v)
and since  $B^*_L F^{n+1}_0 v = 0$.
We have
$A^*_L F_0^{n+1}v \in \mathbb U(r+n,s)$ by Lemma \ref{lem:down2}.
By these comments $0 \not= A^*_L F^{n+1}_0 v \in {\rm ker}(A^*_L) \cap \mathbb U(r+n,s)$.
Therefore $0 \not= {\rm ker}(A^*_L) \cap \mathbb U(r+n,s)$. The map $A^*_L$ is locally nilpotent by Lemma \ref{lem:ln}.
In Appendix A we find $A^*_L A_\ell - q^2 A_\ell A^*_L = I$. The vector space $\mathbb U$ is invariant under $A^*_L$ and $A_\ell$. By these comments, we may apply Appendix E with $S=A^*_L$ and $T = A_\ell$ and $V=\mathbb U$.
By Lemma \ref{lem:step8} the map $A^*_L$ is surjective on $\mathbb U$. By this and Lemma \ref{lem:down2},  $A^*_L \mathbb U(r+n,s) = \mathbb U(r+n-1,s)$. By this and $0 \not= {\rm ker}(A^*_L) \cap \mathbb U(r+n,s)$,
we obtain ${\rm dim}\, \mathbb U(r+n-1,s) < {\rm dim} \, \mathbb U(r+n,s)$. Since $n \in \mathbb N$ is arbitrary, 
\begin{align*}
{\rm dim}\,\mathbb U(r-1,s) <
{\rm dim} \,\mathbb U(r,s) <
{\rm dim}\, \mathbb U(r+1,s) <
{\rm dim}\,\mathbb U(r+2,s) < \cdots 
\end{align*}
This contradicts \eqref{eq:mR} and \eqref{eq:dimu}, so $F_0$ is locally nilpotent.
\medskip

\noindent
Next we consider $F_1$.
Assume that there exists $v \in \mathbb U \cap {\bf V}$ such that $F_1^{m+1}v \not=0$ for all $m \in \mathbb N$. We will get a contradiction.
By our comments below Lemma \ref{lem:partway},
we may assume without loss of generality that $v \in \mathbb U(r,s)$ for some $(r,s) \in \mathbb N^2$. We have $r\geq 1$ and $s\geq 1$, by  Lemma \ref{lem:fff} and $F_1 {\bf 1} = 0$ and $F^3_1 x = 0$.
Let $m \in \mathbb N$. By \eqref{eq:5} and Lemma \ref{lem:uCapv},
 we have $F^{m+1}_1 v \in \mathbb U(r,s+m+1) \cap {\bf V}$. In particular $F^{m+1}_1 v \in {\bf V}$, so 
 $B^*_L F^{m+1}_1 v = 0$ in view of Definition \ref{def:bfv}. 
By these comments $0 \not= F^{m+1}_1 v \in {\rm ker}(B^*_L) \cap \mathbb U(r,s+m+1)$.
Therefore $0 \not= {\rm ker}(B^*_L) \cap \mathbb U(r,s+m+1)$. The map $B^*_L$ is locally nilpotent by Lemma \ref{lem:ln}.
In Appendix A we find $B^*_L B_\ell - q^2 B_\ell B^*_L = I$. The vector space $\mathbb U$ is invariant under $B^*_L$ and $B_\ell$. By these comments, we may apply Appendix E with $S=B^*_L$ and $T = B_\ell$ and $V=\mathbb U$.
By Lemma \ref{lem:step8} the map $B^*_L$ is surjective on $\mathbb U$. By this and Lemma \ref{lem:down2},  $B^*_L \mathbb U(r,s+m+1) = \mathbb U(r,s+m)$. By this and $0 \not= {\rm ker}(B^*_L) \cap \mathbb U(r,s+m+1)$,
we obtain ${\rm dim}\, \mathbb U(r,s+m) < {\rm dim} \, \mathbb U(r,s+m+1)$. Since $m \in \mathbb N$ is arbitrary, 
\begin{align*}
{\rm dim}\,\mathbb U(r,s) <
{\rm dim} \,\mathbb U(r,s+1) <
{\rm dim}\, \mathbb U(r,s+2) <
{\rm dim}\,\mathbb U(r,s+3) < \cdots 
\end{align*}
This contradicts \eqref{eq:mS} and \eqref{eq:dimu}, so $F_1$ is locally nilpotent.
\end{proof}

\noindent The following is our third main result.
\begin{theorem}\label{thm:m3} We have ${\bf U} = \mathbb U \cap {\bf V}$.
\end{theorem}
\begin{proof}
 By Lemmas \ref{lem:eeln}, \ref{lem:ffln} the $U_q(\widehat{\mathfrak{sl}}_2)$-generators $E_0$, $E_1$, $F_0$, $F_1$ are locally nilpotent on
 the $U_q(\widehat{\mathfrak{sl}}_2)$-module $\mathbb U \cap {\bf V}$. Therefore,  the $U_q(\widehat{\mathfrak{sl}}_2)$-module $\mathbb U \cap {\bf V}$ is integrable in the sense of \cite[Definition~4.2]{ariki}.
 By this and \cite[Theorem~3.5.4]{hongkang}, the  $U_q(\widehat{\mathfrak{sl}}_2)$-module $\mathbb U \cap {\bf V}$ is completely reducible. 
 By Lemma \ref{lem:partway}, ${\bf U}$ is a submodule of the $U_q(\widehat{\mathfrak{sl}}_2)$-module $\mathbb U \cap {\bf V}$.
 By these comments, there exists
 a submodule $W$ of the  $U_q(\widehat{\mathfrak{sl}}_2)$-module $\mathbb U \cap {\bf V}$ such that the sum
 $\mathbb U \cap {\bf V} = {\bf U} + W$ is direct.  Assume for the moment that $W\not=0$. Then ${\bf U} \subseteq W$ by Lemma \ref{lem:tp}(i), for a contradiction.
 Consequently $W=0$, so
${\bf U} = \mathbb U \cap {\bf V}$.
\end{proof}

\section{Variations on the theme}

\noindent In Theorem \ref{thm:m1} we turned the vector space $\mathbb U$ into a $U_q(\widehat{\mathfrak{sl}}_2)$-module. In this section we describe three more ways to do this.
Each way yields a $U_q(\widehat{\mathfrak{sl}}_2)$-module $\mathbb U$ that is isomorphic to the one in Theorem \ref{thm:m1}.

\begin{proposition}\label{thm:One} For each row in the table below, the vector space $\mathbb U$ becomes a $U_q(\widehat{\mathfrak{sl}}_2)$-module on which the $U_q(\widehat{\mathfrak{sl}}_2)$-generators
act as indicated:

 \bigskip
 
 \centerline{
 \begin{tabular}{c|ccccccc}
 {\rm generator}&  $E_0$ & $F_0$ & $K^{\pm 1}_0$  & $E_1$ & $F_1$ & $K^{\pm 1} _1$  & $D^{\pm 1}$ 
\\
\hline
{\rm action on $\mathbb U$} & $B^*_R$ & $ \frac{ q B_r K - q^{-1} B_{\ell}}{q-q^{-1}}$ & $q^{\pm 1}K^{\pm 1}$  & $A^*_R$ & $ \frac{ A_r K^{-1}-A_\ell}{q-q^{-1}}$ & $K^{\mp 1}$ &  $Y^{\mp 1}$ 
\\
{\rm action on $\mathbb U$} & $A^*_L$ & $ \frac{ q A_\ell K^{-1} - q^{-1} A_{r}}{q-q^{-1}}$ & $q^{\pm 1}K^{\mp 1}$  & $B^*_L$ & $ \frac{ B_\ell K-B_r}{q-q^{-1}}$ & $K^{\pm 1}$ &  $X^{\mp 1}$ 
\\
{\rm action on $\mathbb U$} & $B^*_L$ & $ \frac{ q B_\ell K - q^{-1} B_{r}}{q-q^{-1}}$ & $q^{\pm 1}K^{\pm 1}$  & $A^*_L$ & $ \frac{ A_\ell K^{-1}-A_r}{q-q^{-1}}$ & $K^{\mp 1}$ &  $Y^{\mp 1}$ 
\end{tabular}
}
\bigskip

\noindent The above three $U_q(\widehat{\mathfrak{sl}}_2)$-modules $\mathbb U$ are isomorphic to the  $U_q(\widehat{\mathfrak{sl}}_2)$-module $\mathbb U$ in Theorem \ref{thm:m1}.  For row 1 (resp. row 2) (resp. row 3),
a $U_q(\widehat{\mathfrak{sl}}_2)$-module isomorphism is given by the restriction of $\sigma$ (resp. $\dagger$) (resp. $\tau$) to $\mathbb U$.
\end{proposition}
\begin{proof}  Below    \eqref{eq:threeDiag} we mentioned that $\mathbb U$ is invariant under each of $\sigma$, $\dagger$, $\tau$. The result follows from this along with Lemmas \ref{lem:Uinv}, \ref{lem:Uinv2}, \ref{lem:obv}  
 and Lemmas \ref{lem:AAXYK}, \ref{lem:AAAABB}, \ref{lem:AAAABB2}.
                                    \end{proof}

 \section{Acknowledgements} 
The author thanks Pascal Baseilhac for many conversations about $U_q(\widehat{\mathfrak{sl}}_2)$ and $U^+_q$.
 
\newpage
\section{Appendix A: Some relations}
In this appendix we list some relations  satisfied by the maps
\begin{align*}
K, \quad K^{-1}, 
\quad
A^*_L, 
\quad
B^*_L, 
\quad
A^*_R, 
\quad
B^*_R, 
\quad
A_\ell, 
\quad
B_\ell, 
\quad
A_r, 
\quad
B_r.
\end{align*}

\begin{proposition} {\rm (See \cite[Proposition~9.1]{boxq}.)} 
\label{thm:v1}
We have
\begin{align*}
&
K A^*_L = q^{-2} A^*_L K,
\qquad \qquad 
K B^*_L = q^2 B^*_L K,
\\
&
K A^*_R = q^{-2} A^*_R K,
\qquad \qquad
K B^*_R = q^{2} B^*_R K,
\end{align*}
\begin{align*}
&
K A_\ell = q^{2} A_\ell K,
\qquad \qquad 
K B_\ell = q^{-2} B_\ell K,
\\
&
K A_r = q^{2} A_r  K,
\qquad \qquad 
K B_r = q^{-2} B_r K,
\end{align*}
\begin{align*}
&
A^*_L A^*_R = A^*_R A^*_L, \qquad \qquad 
B^*_L B^*_R = B^*_R B^*_L,
\\
&
A^*_L B^*_R = B^*_R A^*_L, \qquad \qquad 
B^*_L A^*_R = A^*_R B^*_L,
\end{align*}
\begin{align*}
&
A_\ell A_r = A_r A_\ell, \qquad \qquad 
B_\ell B_r = B_r B_\ell,
\\
&A_\ell B_r = B_r A_\ell, \qquad \qquad 
B_\ell A_r = A_r B_\ell,
\end{align*}
\begin{align*}
&
A^*_L B_r = B_r A^*_L, \qquad  \qquad 
B^*_L A_r = A_r B^*_L, 
\\
& A^*_R B_\ell = B_\ell A^*_R, \qquad 
\qquad
B^*_R A_\ell = A_\ell B^*_R,
\end{align*}
\begin{align*}
&
A^*_L B_\ell = q^{-2} B_\ell A^*_L, \qquad  \qquad
B^*_L A_\ell = q^{-2} A_\ell B^*_L,
\\
&A^*_R B_r = q^{-2} B_r A^*_R, \qquad
\qquad 
B^*_R A_r = q^{-2} A_r B^*_R,
\end{align*}
\begin{align*}
& A^*_L A_\ell - q^2 A_{\ell} A^*_L = I,
\qquad \qquad 
 A^*_R A_r - q^2 A_r A^*_R = I,
\\
& B^*_L B_\ell - q^2 B_{\ell} B^*_L = I,
\qquad \qquad 
 B^*_R B_r - q^2 B_r B^*_R = I,
\end{align*}
\begin{align*}
& 
A^*_L A_r - A_r A^*_L = K,
\qquad \qquad 
B^*_L B_r - B_r B^*_L = K^{-1},
\\
& 
A^*_R A_\ell - A_\ell A^*_R = K,
\qquad \qquad 
B^*_R B_\ell - B_\ell B^*_R = K^{-1},
\end{align*}
\begin{align*}
&A_\ell^3 B_\ell
- 
\lbrack 3 \rbrack_q A_\ell^2 B_\ell A_\ell
+
\lbrack 3 \rbrack_q A_\ell B_\ell A^2_\ell
-
B_\ell A^3_\ell = 0,
\\
&B_\ell^3 A_\ell
- 
\lbrack 3 \rbrack_q B_\ell^2 A_\ell B_\ell
+
\lbrack 3 \rbrack_q B_\ell A_\ell B^2_\ell
-
A_\ell B^3_\ell = 0,
\\
&A_r^3 B_r
- 
\lbrack 3 \rbrack_q A_r^2 B_r A_r
+
\lbrack 3 \rbrack_q A_r B_r A^2_r
-
B_r A^3_r = 0,
\\
&B_r^3 A_r
- 
\lbrack 3 \rbrack_q B_r^2 A_r B_r
+
\lbrack 3 \rbrack_q B_r A_r B^2_r
-
A_r B^3_r = 0.
\end{align*}

\end{proposition}
\begin{proposition} {\rm (See \cite[Proposition~9.3]{boxq}.)}
\label{prop:Uinv}
The following relations hold on $\mathbb U$:
\begin{align*}
&(A^*_L)^3 B^*_L
- 
\lbrack 3 \rbrack_q (A^*_L)^2 B^*_L A^*_L
+
\lbrack 3 \rbrack_q A^*_L B^*_L (A^*_L)^2
-
B^*_L (A^*_L)^3 = 0,
\\
&(B^*_L)^3 A^*_L
- 
\lbrack 3 \rbrack_q (B^*_L)^2 A^*_L B^*_L
+
\lbrack 3 \rbrack_q B^*_L A^*_L (B^*_L)^2
-
A^*_L (B^*_L)^3 = 0,
\\
&(A^*_R)^3 B^*_R
- 
\lbrack 3 \rbrack_q (A^*_R)^2 B^*_R A^*_R
+
\lbrack 3 \rbrack_q A^*_R B^*_R (A^*_R)^2
-
B^*_R (A^*_R)^3 = 0,
\\
&(B^*_R)^3 A^*_R
- 
\lbrack 3 \rbrack_q (B^*_R)^2 A^*_R B^*_R
+
\lbrack 3 \rbrack_q B^*_R A^*_R (B^*_R)^2
-
A^*_R (B^*_R)^3 = 0.
\end{align*}
\end{proposition}

\section{Appendix B: The algebra  $U_q(\widehat{\mathfrak{sl}}_2)$}

In this appendix we recall the quantized  enveloping algebra  $U_q(\widehat{\mathfrak{sl}}_2)$. 
We will generally follow the approach of Ariki \cite[Section~3.3]{ariki}. We will refer to the matrix
\begin{align*}
{\bf A} = \begin{pmatrix} 2 & -2 \\ -2 & 2 \end{pmatrix}.
\end{align*}
We index the rows and columns of $\bf A$ by $0,1$.

\begin{definition}\label{def:Uqhat} \rm (See \cite[Definition~3.16]{ariki}.) Define the algebra $U_q(\widehat{\mathfrak{sl}}_2)$ by generators 
\begin{align*}
 K^{\pm 1}_i, \quad D^{\pm 1}, \quad  E_i, \quad F_i, \qquad \qquad i\in \lbrace 0,1\rbrace
 \end{align*}
 and the following relations. For $i,j\in \lbrace 0,1\rbrace$,
 \begin{align*}
  &K_i K^{-1}_i = K^{-1}_i K_i = 1, \qquad D D^{-1} = D^{-1} D = 1,
  \\
  &\lbrack K_i, K_j\rbrack=0, \qquad \lbrack D, K_i \rbrack=0, 
  \\
  &K_i E_j K^{-1}_i = q^{{\bf A}_{i,j}} E_j, \qquad K_i F_j K^{-1}_i = q^{-{\bf A}_{i,j}} F_j,
  \\
  &DE_0 D^{-1} = q E_0, \qquad D F_0 D^{-1} = q^{-1}F_0, 
  \\
  &\lbrack D, E_1\rbrack = 0, \qquad \lbrack D, F_1\rbrack=0,
  \\
  & \lbrack E_i, F_j \rbrack= \delta_{i,j} \frac{K_i - K^{-1}_i}{q-q^{-1}},
  \\
  &\lbrack E_i, \lbrack E_i, \lbrack E_i, E_j \rbrack _q \rbrack_{q^{-1}} \rbrack = 0,
  \qquad 
  \lbrack F_i, \lbrack F_i, \lbrack F_i, F_j \rbrack_q \rbrack_{q^{-1}} \rbrack = 0, \qquad i \not=j.
\end{align*}
\end{definition}

\begin{note}\rm The Ariki notation is related to our notation as follows.

 \bigskip
 
 \centerline{
 \begin{tabular}{cc}
{\rm Ariki notation} &  {\rm our notation}
\\
\hline
$v$ & $q$\\
$t_i$ & $K_i$ \\
$v^d $ & $D$ \\
$\alpha_j(h_i)$ & ${\bf A}_{i,j}$
\end{tabular}
}
\end{note}

\begin{note}\rm (See \cite[Section~3.3]{ariki}.)
The algebra  $U_q(\widehat{\mathfrak{sl}}_2)$ is sometimes called the quantum algebra of type $A^{(1)}_1$. 
\end{note}

\section{Appendix C: The subspaces  ${\bf U}(r,s)$}
In this appendix, we give a basis for each nonzero ${\bf U}(r,s)$  such that $r+s\leq 10$.

 \bigskip
 
 \centerline{
 \begin{tabular}{cc|c}
$r$& $s$ & basis for ${\bf U}(r,s)$
\\
\hline
$0$&$0$& $\bf 1$ 
\end{tabular}
}
 \bigskip
 
 \centerline{
 \begin{tabular}{cc|c}
$r$& $s$ &   basis for ${\bf U}(r,s)$
\\
\hline
$1$&$0$& $x$ 
\end{tabular}
}
 \bigskip
 
 \centerline{
 \begin{tabular}{cc|c}
$r$& $s$ &  basis for ${\bf U}(r,s)$
\\
\hline
$1$&$1$ & $xy$
\end{tabular}
}
 \bigskip
 
 \centerline{
 \begin{tabular}{cc|c}
$r$& $s$ &  basis for ${\bf U}(r,s)$
\\
\hline
$2$&$1$ & $xyx$
\\
$1$&$2$&$xyy$
\end{tabular}
}
 \bigskip
 
 \centerline{
 \begin{tabular}{cc|c}
$r$& $s$ &   basis for ${\bf U}(r,s)$
\\
\hline
$2$&$2$&$xyxy, xyyx$
\end{tabular}
}
 \bigskip
 
 \centerline{
 \begin{tabular}{cc|c}
$r$& $s$ &  basis for ${\bf U}(r,s)$
\\
\hline
$3$&$2$&$xyxyx$,  $xyyxx$
\\
$2$&$3$&
$xyxyy+xyyxy$
\end{tabular}
}
 \bigskip
 
 \centerline{
 \begin{tabular}{cc|c}
$r$& $s$ &   basis for ${\bf U}(r,s)$
\\
\hline
$4 $&$2$& $xyxyxx+\lbrack 3 \rbrack_q xyyxxx$
\\
$3$&$3$&
$xyxyxy$, $xyyxxy$, $xyyxyx+xyxyyx$
\end{tabular}
}
\bigskip

 \centerline{
 \begin{tabular}{cc|c}
$r$& $s$ &   basis for ${\bf U}(r,s)$
\\
\hline
$4 $&$3$& $xyxyxyx$, $ xyyxyxx+xyxyyxx+xyyxxyx$,
\\
&& $xyxyxxy+ \lbrack 3 \rbrack_q xyyxxxy+ xyyxxyx$
\\
$3$&$4$&
 $xyyxxyy$, $xyyxyxy+xyxyyxy+xyxyxyy$
\end{tabular}
}
\bigskip

 \centerline{
 \begin{tabular}{cc|c}
$r$& $s$ &   basis for ${\bf U}(r,s)$
\\
\hline
$5 $&$3$& $\lbrack 3 \rbrack_q xyyxyxxx+ \lbrack 3 \rbrack_q xyxyyxxx+ \lbrack 2 \rbrack^2_q xyyxxyxx$
\\
&& $ + 2 xyxyxyxx+ xyxyxxyx+ \lbrack 3 \rbrack_q xyyxxxyx$
\\
$4$&$4$&
 $xyxyxyxy$, $xyyxxyyx$, \\
 && $ xyyxyxxy+  xyxyyxxy+ xyyxxyxy$, \\
 && $xyxyxxyy+ \lbrack 3 \rbrack_q xyyxxxyy+ xyyxxyxy$,\\
 && $ xyyxyxyx+ xyxyyxyx+ xyxyxyyx$
\end{tabular}
}
\bigskip

 \centerline{
 \begin{tabular}{cc|c}
$r$& $s$ &  basis for ${\bf U}(r,s)$
\\
\hline
$5 $&$4$& $\lbrack 3 \rbrack_q xyyxyxxxy+ \lbrack 3 \rbrack_q xyxyyxxxy+ \lbrack 2 \rbrack^2_q xyyxxyxxy$
\\
&& $ + 2 xyxyxyxxy+ xyxyxxyxy+ \lbrack 3 \rbrack_q xyyxxxyxy$ \\
&& $ + xyyxyxxyx+xyxyyxxyx+xyyxxyxyx$,
\\ &&$xyxyxyxyx$, $ xyyxxyyxx$,
\\
&& $xyxyxxyyx+ \lbrack 3 \rbrack_q xyyxxxyyx+ xyyxxyxyx$,
\\
&& $xyyxyxyxx+ xyxyyxyxx+ xyxyxyyxx$\\
&&$+ xyyxyxxyx+ xyxyyxxyx+ xyyxxyxyx$
\\
$4$&$5$&
$xyyxyxxyy+xyxyyxxyy+xyyxxyxyy+xyyxxyyxy$,\\
&& $xyxyxyyxy+xyxyyxyxy+xyyxyxyxy+ xyxyxyxyy$, \\
&& $\lbrack 3 \rbrack_q xyxyxxyyy+ \lbrack 3 \rbrack^2_q xyyxxxyyy+ \lbrack 3 \rbrack_q xyyxxyxyy+ xyxyxyxyy$
\end{tabular}
}
\bigskip
 
 \centerline{
 \begin{tabular}{cc|c}
$r$&$s$ &    basis for ${\bf U}(r,s)$
\\
\hline
$6$&$4$&  $\lbrack 3 \rbrack_q xyyxyxxxyx+ \lbrack 3 \rbrack_q xyxyyxxxyx+ \lbrack 2 \rbrack^2_q xyyxxyxxyx$
\\
&& $ + 2 xyxyxyxxyx+ xyxyxxyxyx+ \lbrack 3 \rbrack_q xyyxxxyxyx$ \\
&& $ + xyyxyxxyxx+xyxyyxxyxx+xyyxxyxyxx$
\\
&& $+3 xyxyxyxyxx+ \lbrack 3 \rbrack_q xyyxyxyxxx+\lbrack 3 \rbrack_qxyxyyxyxxx  $ \\
&& $+             \lbrack 3 \rbrack_q xyxyxyyxxx            +      \lbrack 3 \rbrack_q xyyxyxxyxx      +        \lbrack 3 \rbrack_q xyxyyxxyxx    +    \lbrack 3 \rbrack_q xyyxxyxyxx     $,\\
&&$      \lbrack 3 \rbrack_q xyyxxyyxxx+
       xyxyxxyyxx +  \lbrack 3 \rbrack_q xyyxxxyyxx + xyyxxyxyxx$\\
$5$ &$5$& $xyxyxyxyxy$,
\\ &$$& $xyxyxyxyyx+ xyxyxyyxyx+ xyxyyxyxyx+ xyyxyxyxyx$,
\\ &$$& $xyxyxyyxxy+ xyxyyxyxxy+ xyyxyxyxxy+ xyxyyxxyxy$\\
&$$ & $+xyyxxyxyxy+ xyyxyxxyxy$,
\\ &$$ & $xyyxxyyxyx+ xyyxxyxyyx+ xyyxyxxyyx+ xyxyyxxyyx$,
\\&$$ & $xyyxxyyxxy$,
\\
&$$ & $\lbrack 3 \rbrack_q xyyxxyxyyx  + xyxyxxyyxy+ \lbrack 3 \rbrack_q xyyxxxyyxy  +  \lbrack 3 \rbrack^2_q xyyxxxyyyx $
\\
&$$& $+\lbrack 3 \rbrack_q xyxyxxyyyx + xyxyxyxyyx+ xyyxxyxyxy$,
\\& $$& $xyxyyxxyxy+ \lbrack 3 \rbrack_q xyxyyxxxyy + 2 xyxyxyxxyy + xyyxyxxyxy$
\\ &$$&$+ \lbrack 3 \rbrack_q xyyxyxxxyy +  \lbrack 2  \rbrack^2_q xyyxxyxxyy+ 2xyyxxyxyxy + \lbrack 3 \rbrack_q xyyxxxyxyy $
\\ &$$ &$+ \lbrack 3 \rbrack_q xyyxxxyyxy + xyxyxxyxyy+ xyxyxxyyxy$
\\
$4$ & $6$ &$\lbrack 2 \rbrack_q \lbrack 3 \rbrack_q xyyxyxxyyy+\lbrack 2 \rbrack_q \lbrack 3 \rbrack_q xyxyyxxyyy+\lbrack 2 \rbrack_q \lbrack 3 \rbrack_q xyyxxyxyyy$\\
&&$+\lbrack 2 \rbrack_q \lbrack 3 \rbrack_q xyyxxyyxyy +\lbrack 2 \rbrack_q xyxyxyyxyy+\lbrack 2 \rbrack_qxyxyyxyxyy $\\
&& $ +\lbrack 2 \rbrack_qxyyxyxyxyy+\lbrack 2 \rbrack_q \lbrack 3 \rbrack_q xyxyxyxyyy +\lbrack 3 \rbrack_q \lbrack 4 \rbrack_q  xyxyxxyyyy $\\
&& $
+ \lbrack 3 \rbrack^2_q \lbrack 4 \rbrack_q xyyxxxyyyy+ \lbrack 3 \rbrack_q \lbrack 4 \rbrack_q xyyxxyxyyy$
\end{tabular}
}
\bigskip

\section{Appendix D: Some matrix representations}
In this appendix we consider the $U_q(\widehat{\mathfrak{sl}}_2)$-module $\bf U$ from Definition \ref{def:Lambda}.
We display the matrices that represent the actions of $E_0, F_0, K_0, E_1,F_1, K_1, D$ on the bases in Appendix C.
\medskip

\noindent On  ${\bf U}(0,0)$:
\begin{align*}
&K_0: \begin{pmatrix} q \end{pmatrix}, \qquad 
K_1: \begin{pmatrix} 1 \end{pmatrix}, \qquad
D: \begin{pmatrix} 1 \end{pmatrix}.
\end{align*}

\noindent  From  ${\bf U}(1,0)$ to  ${\bf U}(0,0)$:
\begin{align*}
E_0: \begin{pmatrix}
1
\end{pmatrix},
\qquad
E_1: \begin{pmatrix}
0
\end{pmatrix}
\end{align*}

\noindent  From  ${\bf U}(0,0)$ to  ${\bf U}(1,0)$:
\begin{align*}
F_0: \begin{pmatrix}
1
\end{pmatrix},
\qquad
F_1: \begin{pmatrix}
0
\end{pmatrix}
\end{align*}

\noindent On  ${\bf U}(1,0)$:
\begin{align*}
&K_0: \begin{pmatrix} q^{-1} \end{pmatrix}, \qquad 
K_1: \begin{pmatrix} q^2 \end{pmatrix}, \qquad
D: \begin{pmatrix} q^{-1} \end{pmatrix}.
\end{align*}

\noindent  From  ${\bf U}(1,1)$ to  ${\bf U}(1,0)$:
\begin{align*}
E_0: \begin{pmatrix}
0
\end{pmatrix},
\qquad
E_1: \begin{pmatrix}
1
\end{pmatrix}
\end{align*}
\noindent  From  ${\bf U}(1,0)$ to  ${\bf U}(1,1)$:
\begin{align*}
F_0:\begin{pmatrix}
0
\end{pmatrix},
\qquad
F_1:\begin{pmatrix}
\lbrack 2 \rbrack_q
\end{pmatrix}
\end{align*}

\noindent On  ${\bf U}(1,1)$:
\begin{align*}
&K_0: \begin{pmatrix} q \end{pmatrix}, \qquad 
K_1: \begin{pmatrix} 1 \end{pmatrix}, \qquad
D: \begin{pmatrix} q^{-1} \end{pmatrix}.
\end{align*}

\noindent  From  ${\bf U}(2,1)+{\bf U}(1,2)$ to  ${\bf U}(1,1)$:
\begin{align*}
E_0: \begin{pmatrix}
1& 0
\end{pmatrix},
\qquad 
E_1:\begin{pmatrix}0& 1
\end{pmatrix}
\end{align*}

\noindent  From  ${\bf U}(1,1)$ to  ${\bf U}(2,1)+{\bf U}(1,2)$:
\begin{align*}
F_0: \begin{pmatrix}
1\\
 0
\end{pmatrix},
\qquad
F_1:\begin{pmatrix}
0\\
\lbrack 2 \rbrack_q
\end{pmatrix}
\end{align*}

\noindent On  ${\bf U}(2,1)+{\bf U}(1,2)$:
\begin{align*}
&K_0: {\rm diag}(q^{-1},q^3), \qquad
K_1: {\rm diag}(q^2,q^{-2}), \qquad
D:{\rm diag}(q^{-2}, q^{-1}).
\end{align*}

\noindent  From  ${\bf U}(2,2)$ to  ${\bf U}(2,1)+{\bf U}(1,2)$:
\begin{align*}
E_0: \begin{pmatrix}
0& 0\\
 0& 1
\end{pmatrix},
\quad
E_1:\begin{pmatrix}
1& 0\\
 0& 0
\end{pmatrix}
\end{align*}

\noindent  From  ${\bf U}(2,1)+{\bf U}(1,2)$ to  ${\bf U}(2,2)$:
\begin{align*}
F_0: \begin{pmatrix}
0&1\\
0 &\lbrack 3 \rbrack_q
\end{pmatrix},
\quad
F_1:\begin{pmatrix}
\lbrack 2 \rbrack_q & 0 \\
\lbrack 2 \rbrack_q &0
\end{pmatrix}
\end{align*}

\noindent On  ${\bf U}(2,2)$:
\begin{align*}
&K_0: {\rm diag}(q,q), \qquad
K_1: {\rm diag}(1,1), \qquad
D:{\rm diag}(q^{-2}, q^{-2}).
\end{align*}

\noindent  From  ${\bf U}(3,2)+{\bf U}(2,3)$ to  ${\bf U}(2,2)$:
\begin{align*}
E_0: \begin{pmatrix}
1& 0& 0\\
 0 &1& 0
\end{pmatrix},
\quad
E_1:\begin{pmatrix}
0& 0& 1\\
 0& 0& 1
\end{pmatrix}
\end{align*}

\noindent  From  ${\bf U}(2,2)$ to  ${\bf U}(3,2)+{\bf U}(2,3)$:
\begin{align*}
F_0:\begin{pmatrix}
1& 1\\
 0 &\lbrack 2 \rbrack_q^2\\
  0& 0
\end{pmatrix},
\quad
F_1:\begin{pmatrix}
0& 0\\
 0 &0\\
  \lbrack 2 \rbrack_q& 0
\end{pmatrix}
\end{align*}

\noindent On  ${\bf U}(3,2)+{\bf U}(2,3)$:
\begin{align*}
&K_0: {\rm diag}(q^{-1}, q^{-1},q^3), \qquad
K_1: {\rm diag}( q^2,q^2,q^{-2}), \qquad
D:{\rm diag}(q^{-3}, q^{-3}, q^{-2}).
\end{align*}

\noindent  From  ${\bf U}(4,2)+{\bf U}(3,3)$ to  ${\bf U}(3,2)+{\bf U}(2,3)$:
\begin{align*}
E_0: \begin{pmatrix}
1& 0& 0& 0\\
 \lbrack 3 \rbrack_q & 0& 0& 0\\
  0& 0&0& 1
\end{pmatrix},
\quad
E_1: \begin{pmatrix}
0& 1& 0& 0\\
 0& 0& 1& 0\\
  0& 0& 0& 0
\end{pmatrix}
\end{align*}

\noindent  From  ${\bf U}(3,2)+{\bf U}(2,3)$ to  ${\bf U}(4,2)+{\bf U}(3,3)$:
\begin{align*}
F_0: \begin{pmatrix}
0 &1& 0\\
 0& 0& 2 \\
 0& 0 &\lbrack 2 \rbrack_q^2\\
  0& 0 &\lbrack 3 \rbrack_q
\end{pmatrix},
\quad
F_1: \begin{pmatrix}
0& 0& 0\\
 \lbrack 2 \rbrack_q& 0& 0\\
  0& \lbrack 2 \rbrack_q& 0\\
   \lbrack 2 \rbrack_q& 0& 0
\end{pmatrix}
\end{align*}

\noindent On   ${\bf U}(4,2)+{\bf U}(3,3)$:
\begin{align*}
&K_0: {\rm diag}(q^{-3}, q,q,q), \qquad
K_1: {\rm diag}( q^4,1,1,1), \qquad
D:{\rm diag}(q^{-4}, q^{-3}, q^{-3}, q^{-3}).
\end{align*}

\noindent  From  ${\bf U}(4,3)+{\bf U}(3,4)$ to  ${\bf U}(4,2)+{\bf U}(3,3)$:
\begin{align*}
E_0: \begin{pmatrix}
0 &0& 0& 0& 0\\
 1 &0 &0 &0& 0\\
  0 &1& 1& 0& 0\\
   0& 1& 0& 0& 0
\end{pmatrix},
\quad 
E_1: \begin{pmatrix}
0& 0 &1 &0 &0\\
 0 &0 &0 &0 &1\\
  0& 0& 0& 1& 0\\
   0& 0& 0 &0 &1
\end{pmatrix}
\end{align*}
\noindent  From  ${\bf U}(4,2)+{\bf U}(3,3)$ to  ${\bf U}(4,3)+{\bf U}(3,4)$:
\begin{align*}
F_0:\begin{pmatrix}
0 &1& 0& 2\\
 0 &0 &0 &\lbrack 2 \rbrack_q^2\\
  0& 0& 1& 0 \\
  0& 0& 0& 0\\
   0& 0& 0& 0
\end{pmatrix},
\quad 
F_1:\begin{pmatrix}
\lbrack 2 \rbrack_q& 0& 0& 0\\
 \lbrack 2 \rbrack_q& 0& 0& 0\\
  \lbrack 4 \rbrack_q& 0& 0& 0\\
   0& 0& \lbrack 2 \rbrack_q &0\\
    0& \lbrack 2 \rbrack_q& 0& 0
\end{pmatrix}
\end{align*}

\noindent On   ${\bf U}(4,3)+{\bf U}(3,4)$:
\begin{align*}
&K_0: {\rm diag}(q^{-1}, q^{-1},q^{-1},q^3,q^3), \qquad
K_1: {\rm diag}( q^2,q^2,q^2,q^{-2}, q^{-2}), \\
&D:{\rm diag}(q^{-4}, q^{-4}, q^{-4}, q^{-3}, q^{-3}).
\end{align*}

\noindent From ${\bf U}(5,3)+{\bf U}(4,4)$ to  ${\bf U}(4,3)+{\bf U}(3,4)$:
\begin{align*}
E_0: \begin{pmatrix}
2& 0& 0& 0& 0& 0\\
 \lbrack 3 \rbrack_q &0 &0 &0 &0& 0\\
  1 &0 &0& 0& 0& 0 \\
  0& 0& 1& 0& 0& 0 \\
  0& 0& 0& 0& 0& 1
\end{pmatrix},
\quad
E_1: \begin{pmatrix}
0 &1 &0 &0 &0& 0\\
 0 &0 &0 &1& 0& 0\\
  0 &0 &0 &0 &1& 0\\
   0& 0& 0& 0& 0& 0 \\
   0 &0 &0 &0 &0& 0
\end{pmatrix}
\end{align*}
\noindent From ${\bf U}(4,3)+{\bf U}(3,4)$ to  ${\bf U}(5,3)+{\bf U}(4,4)$:
\begin{align*}
F_0: \begin{pmatrix}
0 &1 &0 &0 &0\\
 0 &0 &0& 0& 3 \\
 0 &0& 0& \lbrack 3 \rbrack_q& 0\\
  0& 0& 0& 0& \lbrack 2 \rbrack_q^2\\ 
  0 &0 &0 &1 &0\\
   0 &0 &0 &0& \lbrack 3 \rbrack_q
\end{pmatrix},
\quad
F_1:\begin{pmatrix} 
0 &0 &0 &0 &0\\
 \lbrack 2 \rbrack_q& 0& \lbrack 2 \rbrack_q& 0& 0 \\
 0& \lbrack 2 \rbrack_q& \lbrack 2 \rbrack_q& 0 &0\\
  0 &\lbrack 2 \rbrack_q& \lbrack 2 \rbrack_q&0 &0\\
   0& 0& \lbrack 2 \rbrack_q\lbrack 3 \rbrack_q& 0&0\\
   \lbrack 2 \rbrack_q& 0 &0 &0& 0
\end{pmatrix}
\end{align*}

\noindent On    ${\bf U}(5,3)+{\bf U}(4,4)$:
\begin{align*}
&K_0: {\rm diag}(q^{-3}, q,q,q,q,q), \qquad
K_1: {\rm diag}( q^4,1,1,1,1,1),\\
&D:{\rm diag}(q^{-5}, q^{-4}, q^{-4}, q^{-4}, q^{-4}, q^{-4}).
\end{align*}

\noindent From ${\bf U}(5,4)+{\bf U}(4,5)$ to ${\bf U}(5,3)+{\bf U}(4,4)$:
\begin{align*}
E_0:\begin{pmatrix} 
0& 0& 0& 0& 0& 0& 0& 0\\
 0& 1& 0& 0& 0& 0& 0& 0\\
  0& 0& 1& 0& 0& 0& 0& 0\\
   1& 0& 0& 0& 1& 0& 0& 0\\
    0& 0& 0& 1& 0& 0& 0& 0\\
     0& 0& 0& 0& 1& 0& 0& 0    
\end{pmatrix},
\quad
E_1: \begin{pmatrix} 
1& 0& 0& 0& 0& 0& 0& 0\\
 0& 0& 0& 0& 0& 0& 1& 1\\
  0& 0& 0& 0& 0& 1& 0& 0\\
   0 &0 &0& 0& 0& 1& 0& 0\\
    0& 0& 0& 0& 0& 0& 0& \lbrack 3 \rbrack_q \\
     0 &0 &0 &0 &0 &0 &1 &0
\end{pmatrix}
\end{align*}

\noindent From  ${\bf U}(5,3)+{\bf U}(4,4)$ to  ${\bf U}(5,4)+{\bf U}(4,5)$:
\begin{align*} 
F_0: \begin{pmatrix}
0&0& 0& 1& 0& 0\\
 0& 1& 0& 0& 0& 3\\
  0& 0&  \lbrack 2 \rbrack_q^2& 0& 0& 0\\
   0 &0 &1 &0 &1 &0\\
    0 &0 &0& 0& 0&  \lbrack 2 \rbrack_q^2\\
     0& 0& 0& 0& 0& 0\\
      0& 0& 0 &0 &0 &0\\
       0 &0 &0 &0 &0 &0
\end{pmatrix},
\quad 
 F_1: \begin{pmatrix} 
 \lbrack 4 \rbrack_q & 0& 0& 0& 0& 0\\
  3 \lbrack 2 \rbrack_q& 0& 0& 0& 0& 0\\
   \lbrack 2 \rbrack_q^3& 0& 0& 0& 0& 0\\
    \lbrack 2 \rbrack_q \lbrack 3 \rbrack_q& 0& 0& 0& 0& 0\\ 
    2 \lbrack 2 \rbrack_q& 0& 0& 0& 0& 0\\
     0& 0& 0&  \lbrack 2 \rbrack_q&  \lbrack 2 \rbrack_q& 0\\ 
     0&  \lbrack 2 \rbrack_q& 0 &0 &0 &0\\
      0& 0 &0 &0 & \lbrack 2 \rbrack_q& 0
 \end{pmatrix}
\end{align*}

\noindent On   ${\bf U}(5,4)+{\bf U}(4,5)$:
\begin{align*}
&K_0: {\rm diag}(q^{-1}, q^{-1}, q^{-1},q^{-1},q^{-1},q^3,q^3,q^3), \qquad
K_1: {\rm diag}( q^2, q^2, q^2,q^2, q^2,q^{-2}, q^{-2}, q^{-2}),\\
&D:{\rm diag}(q^{-5}, q^{-5}, q^{-5}, q^{-5}, q^{-5}, q^{-4}, q^{-4}, q^{-4}).
\end{align*}

\noindent From ${\bf U}(6,4)+{\bf U}(5,5)+{\bf U}(4,6)$ to ${\bf U}(5,4)+{\bf U}(4,5)$:
 \begin{align*}
 E_0:  \begin{pmatrix}
 1& 0& 0& 0& 0& 0& 0& 0& 0& 0\\
  3& 0& 0& 0& 0& 0& 0& 0& 0& 0\\
   0& \lbrack 3 \rbrack_q & 0& 0& 0& 0& 0& 0& 0& 0\\
    0& 1& 0& 0& 0& 0& 0& 0& 0& 0 \\
    \lbrack 3 \rbrack_q& 0& 0& 0& 0& 0& 0& 0& 0& 0\\
     0& 0& 0& 0& 0& 1& 0& 0& 0& 0\\
      0& 0& 0& 1& 0& 0& 0& 0& 0& 0\\
       0& 0& 0& 0& 0& 0& 0& 1& 0& 0\\
       \end{pmatrix},
       \quad E_1:
       \begin{pmatrix}
       0& 0& 0& 0& 0& 0& 0& 0& 1& 0 \\
       0& 0& 1& 0& 0& 0& 0& 0& 0& 0\\
        0& 0& 0& 0& 0& 0& 1& 0& 0& 0 \\
        0& 0& 0& 0& 0& 0& 0& 1& 1& 0 \\
        0& 0& 0& 0& 1& 0& 0& 0& 0& 0\\ 
        0& 0& 0& 0& 0& 0& 0& 0& 0& \lbrack 2 \rbrack_q \lbrack 3 \rbrack_q \\
        0& 0& 0& 0& 0& 0& 0& 0& 0& \lbrack 2 \rbrack_q  \\
        0&0& 0& 0& 0& 0& 0& 0& 0& \lbrack 4 \rbrack_q \\
       \end{pmatrix}
 \end{align*}
 \noindent From  ${\bf U}(5,4)+{\bf U}(4,5)$ to ${\bf U}(6,4)+{\bf U}(5,5)+{\bf U}(4,6)$:
 \begin{align*}
 F_0: \begin{pmatrix}
 0& 0& 0& 0& 1& 0 &0 &0 \\
 0& 0&1& 0& 0& 0& 0& 0\\
  0& 0& 0& 0&0& 0& 4& 1\\
   0& 0& 0& 0& 0& 0& \lbrack 3 \rbrack_q & 0\\
   0& 0& 0& 0& 0& 0& \lbrack 2 \rbrack^2_q & 0\\
   0& 0& 0& 0& 0& \lbrack 3 \rbrack_q & 0& 0\\ 
   0& 0& 0& 0& 0& \lbrack 2 \rbrack^2_q & 0& 0\\
    0& 0& 0& 0& 0& 0& 0& \lbrack 3 \rbrack_q \\
    0& 0& 0& 0& 0& 1& 0& 0\\
     0& 0& 0& 0& 0& 0& 0& 0 \end{pmatrix},
 \quad
F_1: \begin{pmatrix} 
0& 0& 0& 0& 0& 0& 0& 0\\
 0& 0& 0& 0& 0& 0& 0& 0\\
 3 \lbrack 2 \rbrack_q & \lbrack 2 \rbrack_q & 0& 0& 0& 0& 0& 0\\
   0& \lbrack 2 \rbrack_q & 0& 0& 0& 0& 0& 0 \\
   2 \lbrack 2 \rbrack_q  &0 &0& 0&\lbrack 2 \rbrack_q  & 0& 0& 0 \\
   \lbrack 2 \rbrack_q& 0& 0& \lbrack 2 \rbrack_q&\lbrack 2 \rbrack_q& 0& 0& 0\\
    \lbrack 2 \rbrack_q^3 &0& \lbrack 2 \rbrack_q& 0& 0& 0& 0& 0\\
     0& 0& 0& \lbrack 2 \rbrack_q& 0& 0& 0& 0 \\
     \lbrack 2 \rbrack_q\lbrack 3 \rbrack_q& 0& 0& 0& 0& 0& 0& 0\\
      0& 0& 0& 0& 0& 0& 0& 1\end{pmatrix}
\end{align*}

\noindent On  ${\bf U}(6,4)+{\bf U}(5,5)+{\bf U}(4,6)$:
\begin{align*}
&K_0: {\rm diag}(q^{-3}, q^{-3}, q,q,q,q,q,q,q,q^5), \qquad
K_1: {\rm diag}( q^4, q^4, 1,1,1,1,1,1,1,q^{-4}),\\
&D:{\rm diag}(q^{-6}, q^{-6}, q^{-5}, q^{-5}, q^{-5}, q^{-5}, q^{-5}, q^{-5}, q^{-5}, q^{-4}).
\end{align*}

\section{Appendix E: Some linear algebra}

In this appendix we consider the following situation. Let $V$ denote an infinite-dimensional vector space. Let $S: V \to V$ and $T: V \to V$ denote $\mathbb F$-linear maps. Assume that
$S$ is locally nilpotent and
\begin{align}
ST - q^2 TS = I.
\label{eq:Weyl}
\end{align}       
We will  show that $S$ is surjective and $T$ is injective. We remark that the surjectivity of $S$ is used in the proof of Lemma \ref{lem:ffln}, and the injectivity of $T$ is used to obtain
the surjectivity of $S$.
\begin{lemma} \label{lem:step1} The map $T$ is injective.
\end{lemma}
\begin{proof} Let $v \in V$ such that $Tv=0$. We show that $v=0$. For $n\geq 1$, use \eqref{eq:Weyl} and induction on $n$ to obtain
\begin{align}
T S^n v = -q^{-n-1} \lbrack n \rbrack_q S^{n-1} v.
\label{eq:back}
\end{align}
Since $S$ is locally nilpotent, there exists $n\geq 1$ such that $S^nv=0$.
By applying  \eqref{eq:back} repeatedly, we see that each of $S^nv, S^{n-1}v, \ldots Sv, v$ is equal to 0. In particular $v=0$.
\end{proof}

\noindent For $n \in \mathbb N$, we adjust \eqref{eq:Weyl} to obtain
\begin{align}\label{eq:fact}
ST - q^n \lbrack n+1 \rbrack_q I= q^2 \bigl(TS - q^{n-1} \lbrack n \rbrack_q I\bigr).
\end{align}
Therefore, the kernel of $ST - q^n \lbrack n+1 \rbrack_q I$ is equal to the kernel of $TS - q^{n-1} \lbrack n \rbrack_q I$.  Let $V_n$ denote this common kernel.
By construction
\begin{align}
\bigl( ST - q^n \lbrack n+1 \rbrack_q I \bigr) V_n = 0, \qquad \qquad 
\bigl(TS - q^{n-1} \lbrack n \rbrack_q I \bigr) V_n = 0.
\label{eq:twoV}
\end{align}
Note that the sum $\sum_{n \in \mathbb N} V_n$ is direct.
 For notational convenience define $V_{-1}=0$.

\begin{lemma}\label{lem:stepA} We have ${\rm ker} (S)=V_0$.
\end{lemma}
\begin{proof} By the discussion below \eqref{eq:fact}, we obtain ${\rm ker} (TS)= V_0$. The map $T$ is injective by Lemma \ref{lem:step1}, so ${\rm ker}(S) = {\rm ker}(TS)$. Therefore ${\rm ker}(S)=V_0$.
\end{proof}

\begin{lemma} \label{lem:step3} For $n \in \mathbb N$ we have
\begin{align*}
S V_n \subseteq V_{n-1}, \qquad \qquad T V_n \subseteq V_{n+1}.
\end{align*}
\end{lemma}
\begin{proof} First we verify $SV_n \subseteq V_{n-1}$. For $n=0$ this holds by Lemma \ref{lem:stepA}. For $n\geq 1$ we 
use \eqref{eq:twoV} to obtain
\begin{align*}
\bigl(ST-q^{n-1} \lbrack n \rbrack_q I \bigr) SV_n = S \bigl(TS - q^{n-1} \lbrack n \rbrack_q I\bigr) V_n = S 0  = 0,
\end{align*}
so $SV_n \subseteq V_{n-1}$. Next we verify $T V_n \subseteq V_{n+1}$. For $n\geq 0$ we have
\begin{align*}
\bigl(TS-q^{n} \lbrack n+1 \rbrack_q I \bigr) TV_n = T \bigl(ST - q^{n} \lbrack n+1 \rbrack_q I\bigr) W_n = T0  = 0,
\end{align*}
so $TV_n \subseteq V_{n+1}$. 
\end{proof}

\begin{lemma}\label{lem:step4} For $n\geq 1$ the following  maps are inverses:
\begin{align}
\label{eq:inv1}
S: V_n \to V_{n-1}, \qquad \qquad q^{1-n} \lbrack n \rbrack^{-1}_q T: V_{n-1} \to V_n.
\end{align}
\end{lemma}
\begin{proof} 
 By \eqref{eq:twoV} we have
$(S T_n-I)V_{n-1}=0$ and $(T_n S-I)V_n=0$, where  $T_n = q^{1-n} \lbrack n \rbrack^{-1}_q T$.
\end{proof}

\begin{lemma} \label{lem:step5} For $n\geq 1$ the maps
\begin{align*}
S: V_n \to V_{n-1}, \qquad \qquad T: V_{n-1} \to V_n
\end{align*}
are bijections.
\end{lemma}
\begin{proof} By Lemma \ref{lem:step4}.
\end{proof}

\begin{lemma} \label{lem:step6} For $n \in \mathbb N$,
\begin{align} \label{eq:kneed}
{\rm ker} (S^{n+1}) = V_0+ V_1 + \cdots + V_n.
\end{align}
\end{lemma} 
\begin{proof} We use induction on $n$. First assume that $n=0$. Then \eqref{eq:kneed} holds by Lemma \ref{lem:stepA}.
Next assume that $n\geq 1$.
The inclusion $\supseteq $ in \eqref{eq:kneed} holds by Lemma \ref{lem:step3}. We next obtain the inclusion $\subseteq $ in \eqref{eq:kneed}.
Let $v \in {\rm ker} (S^{n+1})$. We will show that $v \in V_0+ V_1 + \cdots + V_n$.
We have $0 = S^{n+1} v = S^{n} Sv$, so by induction $Sv \in V_0+V_1+ \cdots + V_{n-1}$.
By Lemma \ref{lem:step5}, there exists $w \in V_1+ V_2 + \cdots + V_n$ such that $Sw = Sv$. Therefore $S(w-v)=0$, so $w -v \in V_0$.
By these comments $v \in V_0 + V_1 + \cdots + V_n$.
\end{proof}
\begin{lemma} \label{lem:step7} We have $V=\sum_{n \in \mathbb N} V_n$.
\end{lemma}
\begin{proof} Since $S$ is locally nilpotent, we have $V = \cup_{n \in \mathbb N} {\rm ker} (S^{n+1})$. The result follows from this and Lemma \ref{lem:step6}.
\end{proof}
\begin{lemma} \label{lem:step8}
The map $S$ is surjective.
\end{lemma}
\begin{proof} By Lemma \ref{lem:step7}  and since $V_n = S V_{n+1}$ for $n \in \mathbb N$.
\end{proof}
                       


%

\bigskip

\noindent Paul Terwilliger \hfil\break
\noindent Department of Mathematics \hfil\break
\noindent University of Wisconsin \hfil\break
\noindent 480 Lincoln Drive \hfil\break
\noindent Madison, WI 53706-1388 USA \hfil\break
\noindent email: {\tt terwilli@math.wisc.edu }\hfil\break

\end{document}